\documentclass[12pt]{amsart}

\usepackage{xcolor}
\usepackage{tikz}
\usepackage{accents}
\usepackage{graphicx, psfrag}
\usepackage{dsfont}
\usepackage{amssymb, dsfont, amsthm, amsmath}
\usepackage{pgf}
\usepackage{geometry}

\newtheorem{theorem}{Theorem}[section]

\newtheorem{lemma}[theorem]{Lemma}
\newtheorem{proposition}[theorem]{Proposition}
\newtheorem{definition}[theorem]{Definition}
\theoremstyle{remark}
\newtheorem{remark}[theorem]{Remark}
\newtheorem{example}[theorem]{Example}

\newcommand{\set}[1]{\left\{#1\right\}}
\newcommand{\trr}{\triangleright}
\newcommand{\brr}{\blacktriangleright}
\newcommand{\rrt}{\triangleleft\,}
\newcommand{\Hash}{\begin{minipage}{8pt}\includegraphics[width=8pt]{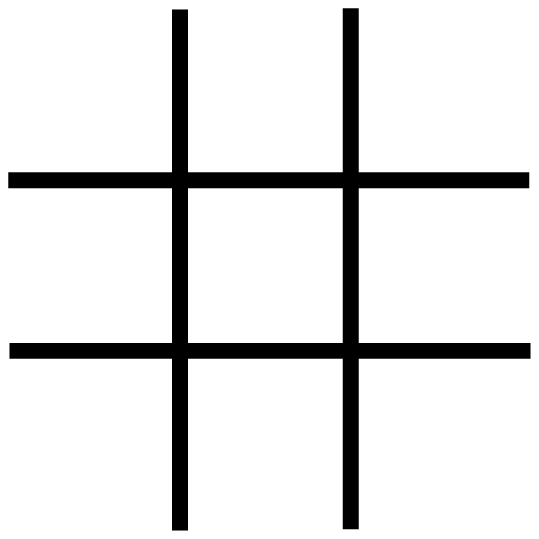}\end{minipage}}

\newcommand{\ass}{\stackrel{\textup{\tiny def}}{=}}

\newcommand{\dU}{ \, \raisebox{1.25pt}{\scalebox{.8}{$\coprod$}} \; }

\allowdisplaybreaks

\begin{document}



\title{Inca Foams}

\author{Avishy Y. Carmi and Daniel Moskovich}

\address{Faculty of Engineering Sciences \& \\ Center for Quantum Information Science and Technology \\ Ben-Gurion University of the Negev, Beer-Sheva 8410501, Israel}

\begin{abstract}
We study a certain class of embedded two-foams that arise from gluing discs into ribbon torus knots along nonintersecting torus meridians. We exhibit several equivalent diagrammatic formalisms for these objects and identify several of their invariants, including a unique prime decomposition.
\end{abstract}


\maketitle

\section{Introduction}


An \emph{embedded $2$--foam in standard Euclidean $\mathds{R}^4$} is a $2$--dimensional analogue of a knotted trivalent graph $\mathds{R}^3$. We investigate a certain class of embedded $2$--foams which we call \emph{Inca foams} (this name was suggested to us from \cite{Siegel:15}) that arise from gluing discs into ribbon torus knots along nonintersecting torus meridians. We exhibit five diagrammatic formalisms for Inca foams. We then identify several invariants of Inca foams, including their unique prime decompositions.

The geometric topological study of embedded $2$--foams was initiated by Carter \cite{Carter:12}. A knotted surface in $\mathds{R}^4$ is in particular a $2$--foam, and so the theory of embedded $2$--foams is at least as complicated the theory of knotted surfaces in dimension $4$ \cite{CarterKamadaSaito:04}.
Inca foams are a much simpler class of objects than $2$--foams (\textit{e.g.} no local knotting) but are still complicated enough to be interesting. The theory of coloured versions of such foams is equivalent to a construction that was used by the authors to topologically model fusion of information and as a model for computation \cite{CarmiMoskovich:14,CarmiMoskovich:15a,CarmiMoskovich:15b}, and it is this that is our main motivation.

We describe the contents of this paper. Section~\ref{S:Necklace2Foam} defines Inca foams which Section~\ref{S:FiveDescriptions} describes in five (5) different diagrammatic ways. The first is a broken surface diagram of the foam, the second is a broken surface diagram of tangled spheres and intervals, the third is a $3$--dimensional analogue of a tangle diagram, the fourth is a tangle diagram, and the fifth is a Gau{\ss} diagram. Section~\ref{S:FormalismEquivalence} proves that these are equivalent. Each of the diagrammatic formalisms is useful for something else. The more topological ones are better for proving theorems, and the more combinatorial ones are better for defining invariants. 

Section~\ref{S:Invariants} describes some simple Inca foam invariants. Some, such as underlying graph and underlying w-tangle, are structures which we obtain by suppressing some of the information in an Inca foam. Some, such as the fundamental quandle and the linking graph, are analogues of classical link invariants. One, the Shannon capacity, is an analogue of a graph invariant.

Section~\ref{S:PrimeDecomposition} discusses unique prime decomposition for connected Inca foams. The prime decomposition is of course an invariant up to permutation of factors. Existence and uniqueness of prime decomposition indicates how well-behaved this class of objects is in comparison with classes such as virtual tangles and w-tangles.

Much of the content of this paper originally appeared in the preprint \cite{CarmiMoskovich:14b}, which is being split into parts, the first of which is the present paper.

\section{Inca foams}\label{S:Necklace2Foam}

So what is an Inca foam? We give the definition below. 

\begin{definition}\label{D:Necklace}
Parameterize $S^1\simeq \mathds{R}/\mathds{Z}$, and for a given $k\geq 1$, glue $2$--discs $D_1,\ldots,D_k$ to a torus $T^2\ass S^1\times S^1$ so that $\partial D_j$ glues to $S^1\times \set{\frac{j}{k}}$ for $j=1,2,\ldots,k$. Call the resulting shape $K_k$. An \emph{Inca foam} is an immersion
\[F\colon\, \bigcup_{i=1}^\nu K_{k_i}\to S^4\simeq \mathds{R}^4\cup\set{\infty}\enspace ,\]
\noindent whose restriction to $\mathds{R}^4$ is an embedding, for which $F(T^2)$ bounds a solid torus for each $K_{k_i}$. In the above, $k_1,k_2,\ldots,k_\mu$ are positive integers. Call $F(K_{k_1}), F(K_{k_2}),\ldots, F(K_{k_\nu})$ the \emph{components} of $F$. See Figure~\ref{F:NecklaceExample}.

Equivalently (and to fix notation), $F(D_i\cup (S^1\times [\frac{i}{n},\frac{i+1}{n}])\cup D_{i+1})$ is homeomorphic to a sphere $S_i$, and this sphere must bound a $3$--ball $B_i$ in $S^4$.

We additionally require that the special point $\set{\infty}\in S^4$ lie either outside the bounded solid tori, or else that it lie in the interior of at most one of the $2$--discs for each connected component. Thus, different connected components might intersect but only inside their $2$--discs and only at the point $\infty$.
\end{definition}

\begin{figure}
\centering
\includegraphics[width=0.5\textwidth]{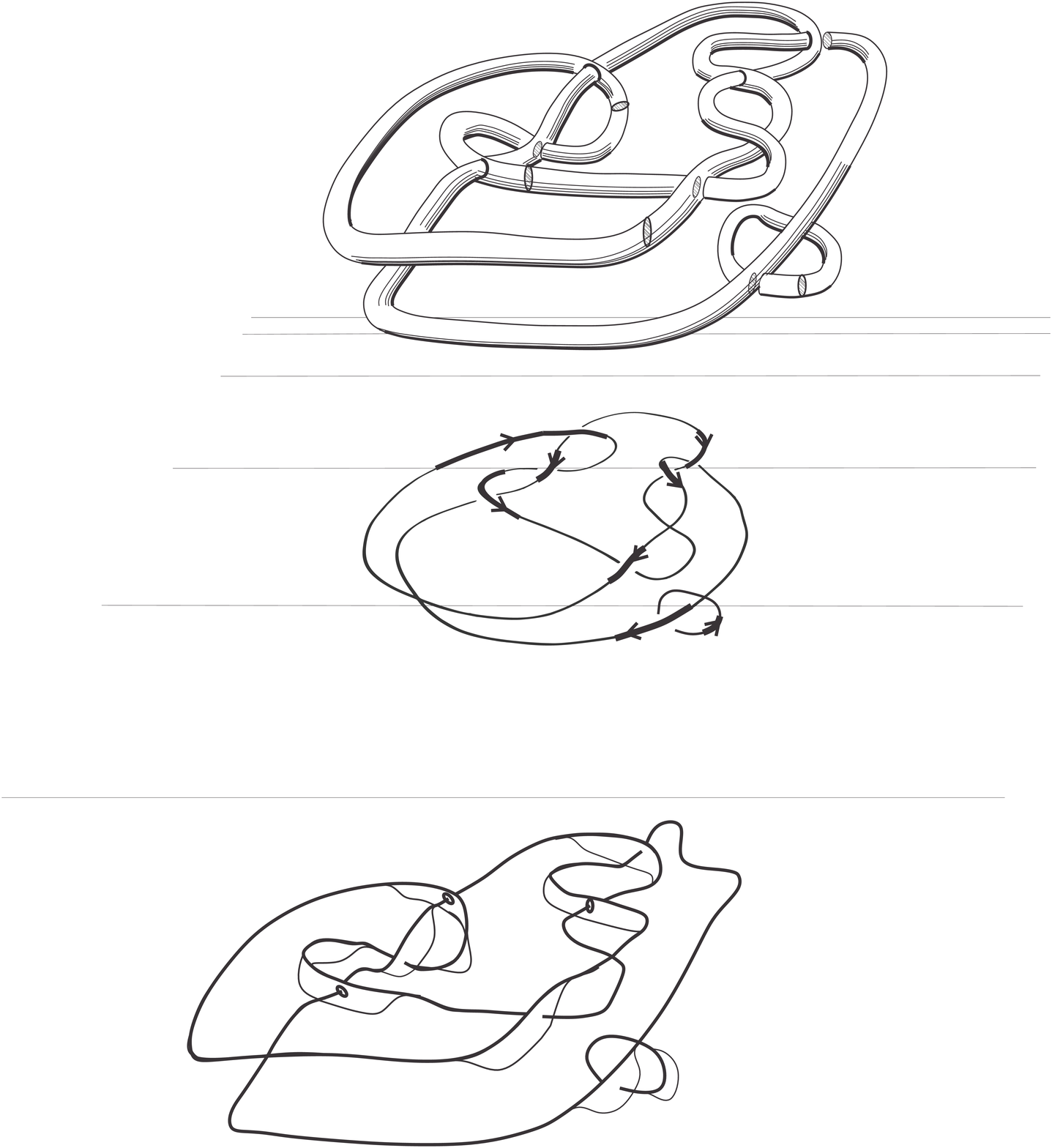}
\caption{\label{F:NecklaceExample} A Roseman, a Rosemeister, and a Reidemeister diagram of an Inca foam.}
\end{figure}

Our convention is that our objects live in the smooth category, and we smooth corners automatically at every step usually without comment. Such sloppiness is standard in geometric topology--- \cite{Kirby:89} famously begins with the words ``\ldots the phrase ``corners can be smoothed'' has been a phrase that I have heard for 30 years, and this is not the place to explain it''.

Various generalizations of Definition~\ref{D:Necklace} suggest themselves. For example, if we allow each sphere to intersect an arbitrary number of other spheres at disjoint disks, \textit{i.e.} if we consider surfaces of higher genus than tori, then the effect is only to make diagrams and notations more complicated--- the mathematics is essentially unchanged and all of our constructions generalize in a straightforward way. For example, the underlying graph of a Gau{\ss} diagram (see Section~\ref{SS:Gauss}) becomes an arbitrary graph instead of a collection of paths and cycles. If disks are allowed to have disk intersection then trees replace intervals \textit{e.g.} in Section~\ref{SS:SITangle}, and underlying graphs have two different kinds of vertices; Everything generalizes to this setting as well but more work is needed.

Inca foams are considered equivalent if they are \emph{ambient isotopic in $\mathds{R}^4$}, a definition which we recall below in our setting:

\begin{definition}[Ambient isotopy]
Consider a class $\mathfrak{T}$ of embedded objects in standard $S^4\simeq \mathds{R}^4\cup \{\infty\}$.
Two embedded objects $T_1, T_2\in \mathfrak{T}$ are \emph{ambient isotopic in $\mathds{R}^4$} if there exists a smooth homeomorphism $h\colon\, \mathds{R}^4\times [0,1]\to \mathds{R}^4$ with $h(T_1\times \{0\})=T_1$, and $h(T_1\times \{t\})$ is an element of $\mathfrak{T}$ for all $t\in[0,1]$, and $h(T_1\times \{1\})=T_2$.
\end{definition}

We further define \emph{(de)stabilization} to be the following operation.

\begin{definition}[Stabilization; Destabilization]
Let $S$ be a sphere in an Inca foam $F$ which bounds a ball $B$ whose interior does not intersect $F$. Let $S^\prime$ be a sphere in $F$ which shares a disk $D$ with $S$. The \emph{destabilization of $F$ by $D$} is the Inca foam obtained by erasing $D$ from $F$ and smoothing corners (so that $S$ and $S^\prime$ effectively become a single sphere). The inverse operation to destabilization is called \emph{stabilization}. See Figure~\ref{F:FoamStabilize}. 
\end{definition}

\begin{figure}
\centering
$\begin{minipage}{1cm}\includegraphics[height=3cm]{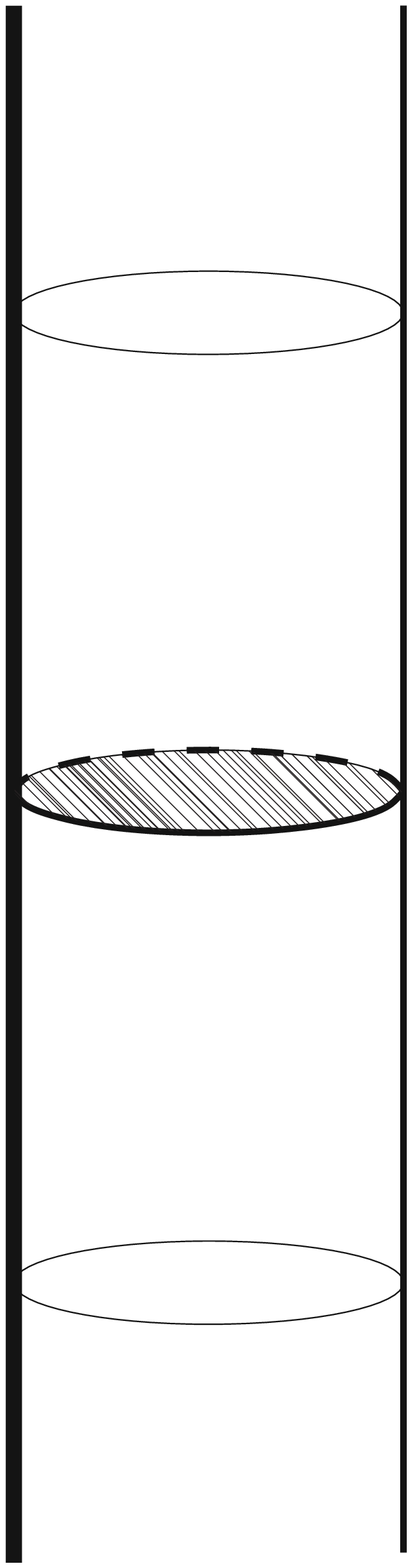}\end{minipage}\quad\longleftrightarrow\quad \begin{minipage}{1cm}\includegraphics[height=3cm]{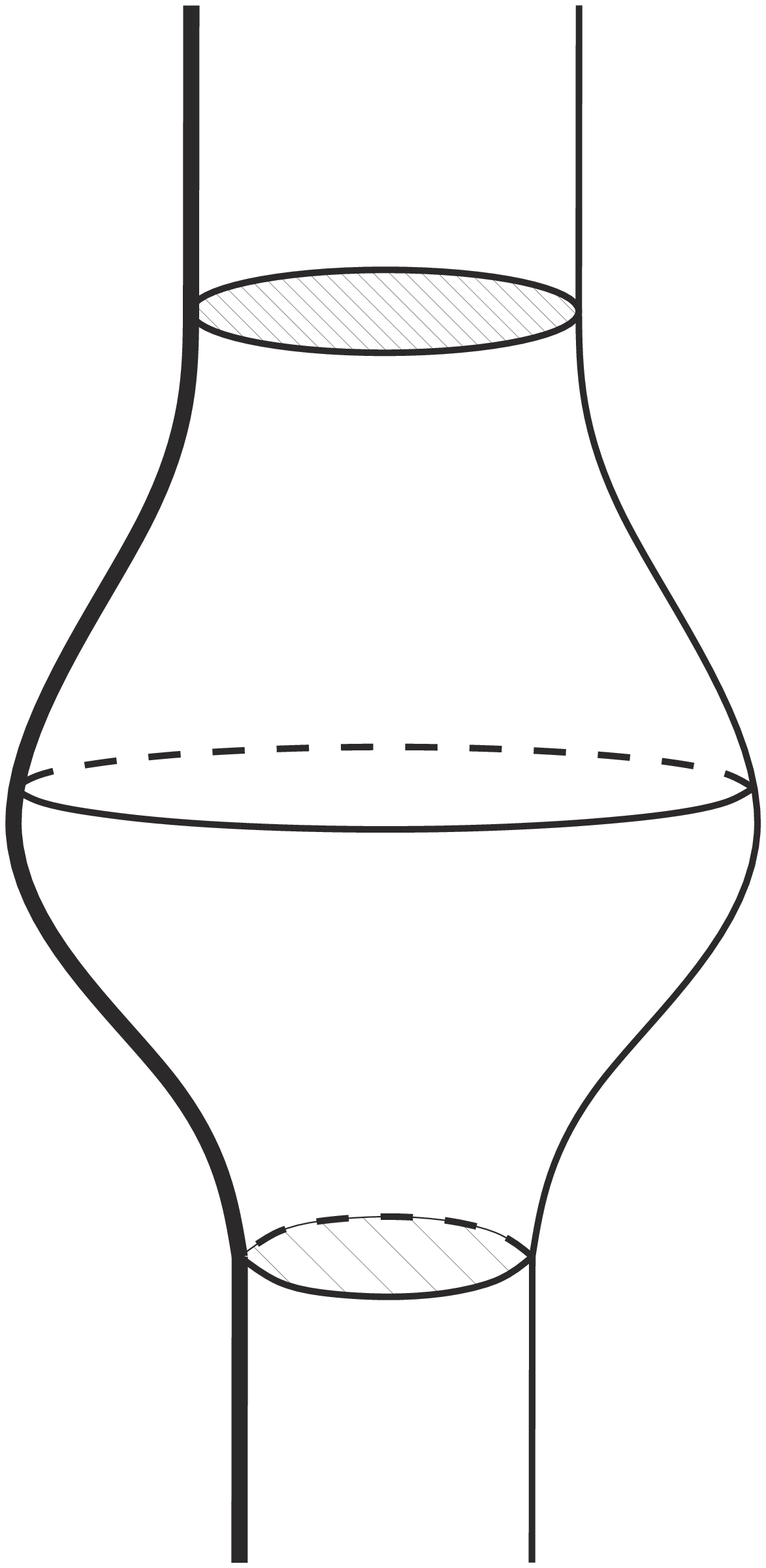}\end{minipage}$
\caption{\label{F:FoamStabilize}Stabilization of Inca foams.}
\end{figure}

\begin{definition}[Equivalence; Stable equivalence]
Two Inca foams are said to be \emph{equivalent} if they are ambient isotopic. They are said to be \emph{stably equivalent} if they have equivalent stabilizations.
\end{definition}

\section{Five diagrammatic descriptions}\label{S:FiveDescriptions}

This section describes Inca foams in five different ways, starting from the most geometric and progressing to the most combinatorial. Section~\ref{S:FormalismEquivalence} proves that these describe the same objects. The more geometric descriptions are easier to use to prove theorems with, while the more combinatorial ones are better suited for calculations.

\begin{table}
\begin{center}
\begin{tabular}{ | c | c| c | c | c |}
\hline
Formalism & Section & Agent & Local moves & Stabilization\\ \hline
Inca foams \rule{0pt}{1.25cm} & \ref{SS:Roseman} & \begin{minipage}{0.4cm}
\includegraphics[height=1.3cm]{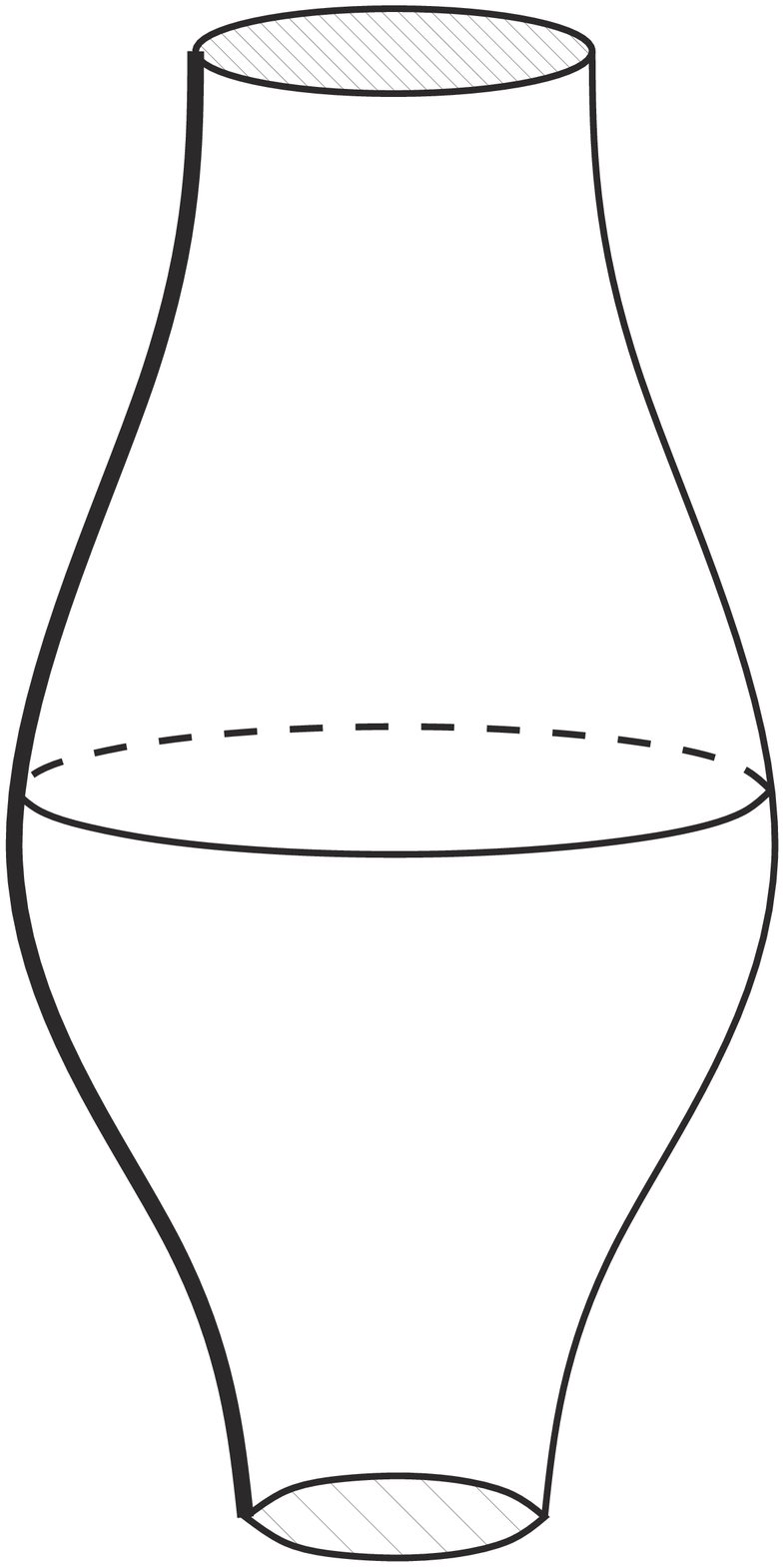}
\end{minipage} & \parbox[c]{1.35in}{Carter--Roseman moves \cite{Carter:12}.}& $\quad \begin{minipage}{0.4cm}\includegraphics[height=1.2cm]{foam_stabilize2}\end{minipage}\longleftrightarrow \begin{minipage}{0.4cm}\includegraphics[height=1.2cm]{foam_stabilize1}\end{minipage}\quad $\\[0.82cm]
\hline
Roseman & \ref{SS:SITangle} & \begin{minipage}{0.3cm}
\includegraphics[height=1.8cm]{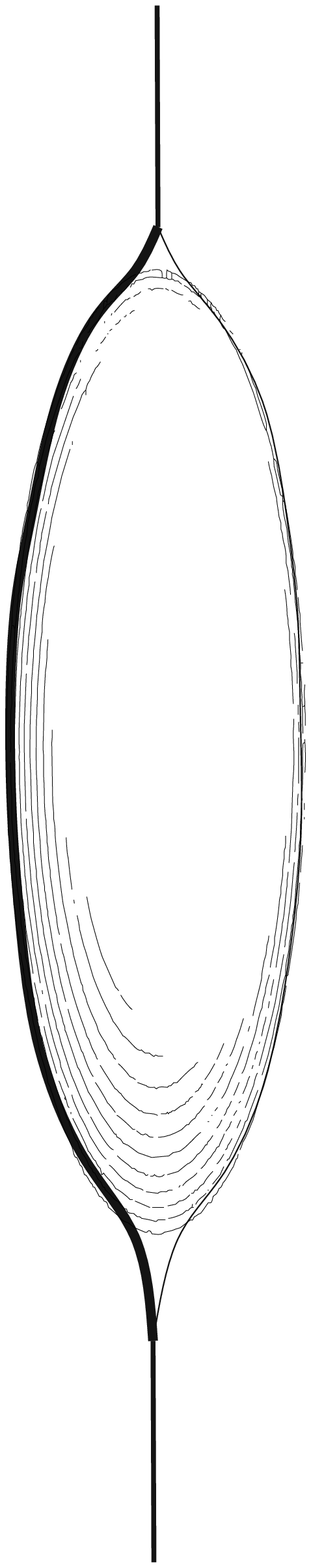}
\end{minipage} & \scalebox{0.21}{\begin{minipage}{0.95\textwidth}
\centering
\includegraphics[width=\textwidth]{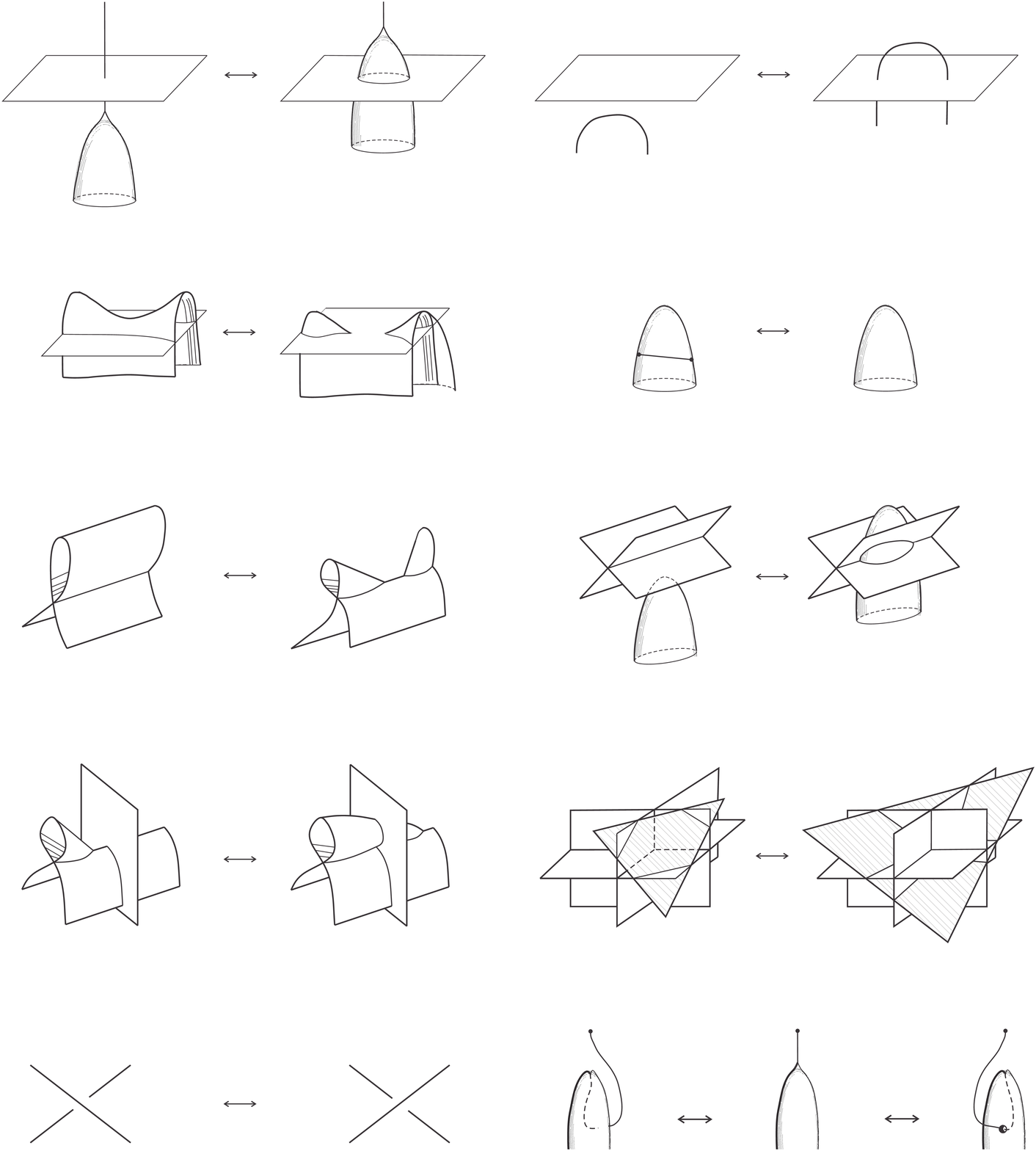}\end{minipage}}\rule{0pt}{2.4cm} & $\begin{minipage}{0.4cm}\ \,\includegraphics[height=1.2cm]{vline}\end{minipage}\rule{0pt}{1.8cm}\longleftrightarrow\  \begin{minipage}{0.4cm}\includegraphics[height=1.4cm]{si_agent}\end{minipage}$\\[2.2cm]
\hline
Rosemeister & \ref{SS:Rosemeister} & \begin{minipage}{0.3cm}
\includegraphics[height=1.8cm]{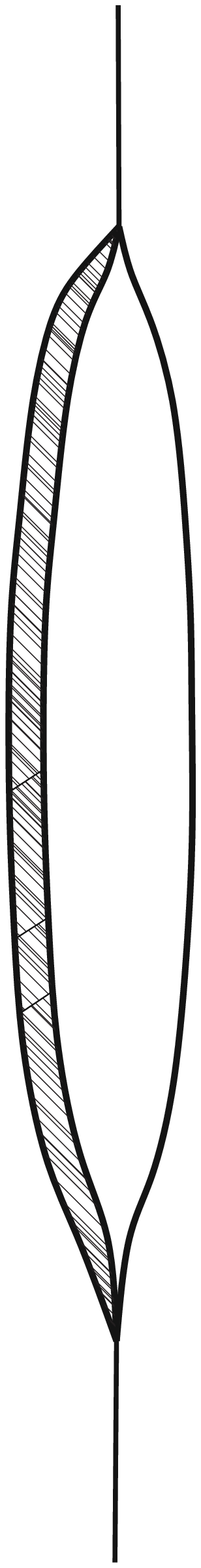}
\end{minipage} & \scalebox{0.3}{\begin{minipage}{0.95\textwidth}\psfrag{r}[c]{}\psfrag{s}[c]{}\psfrag{t}[c]{}\psfrag{u}[c]{}
\centering
\includegraphics[width=0.95\textwidth]{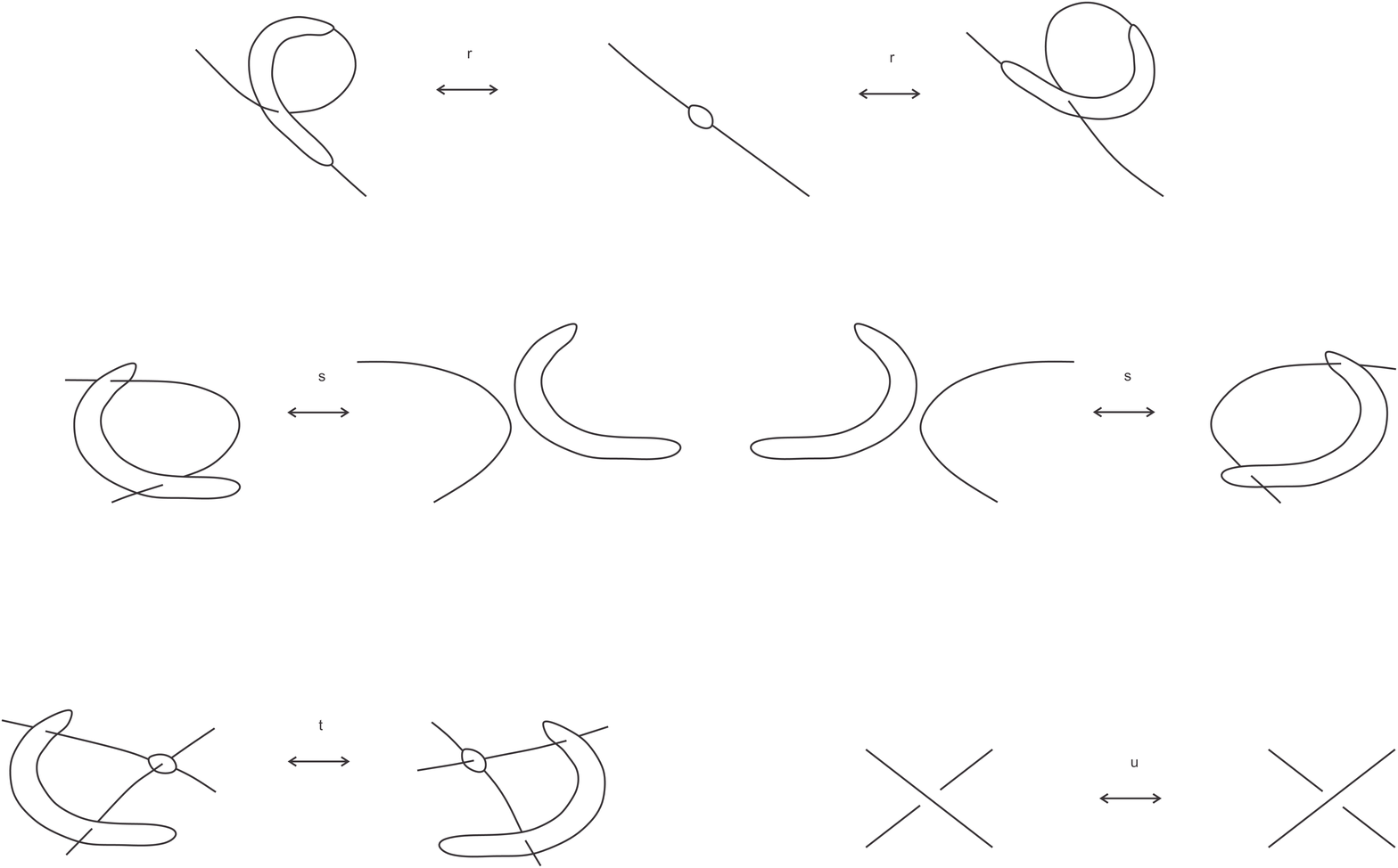}\end{minipage}}\rule{0pt}{1.8cm} & $\begin{minipage}{0.4cm}\ \,\includegraphics[height=1.2cm]{vline}\end{minipage}\longleftrightarrow\  \begin{minipage}{0.4cm}\includegraphics[height=1.4cm]{rosemeister_agent}\end{minipage}$\\[1.8cm]
\hline
Reidemeister & \ref{SS:Tangle} & \begin{minipage}{0.3cm}
\includegraphics[height=1.8cm]{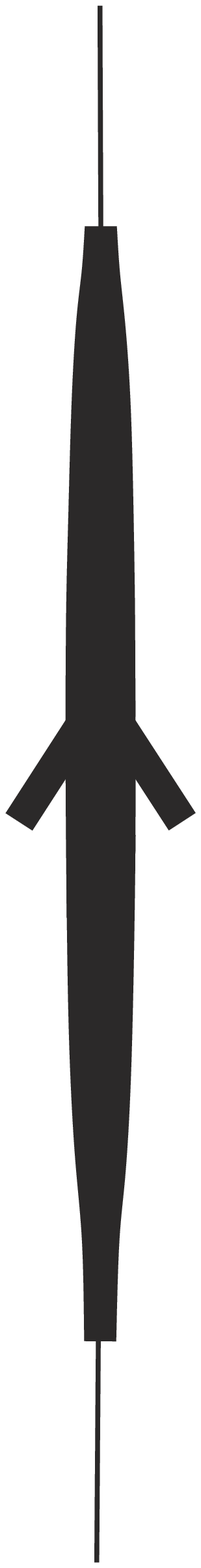}
\end{minipage} & \scalebox{0.25}{\begin{minipage}{0.95\textwidth}
\centering
\includegraphics[width=\textwidth]{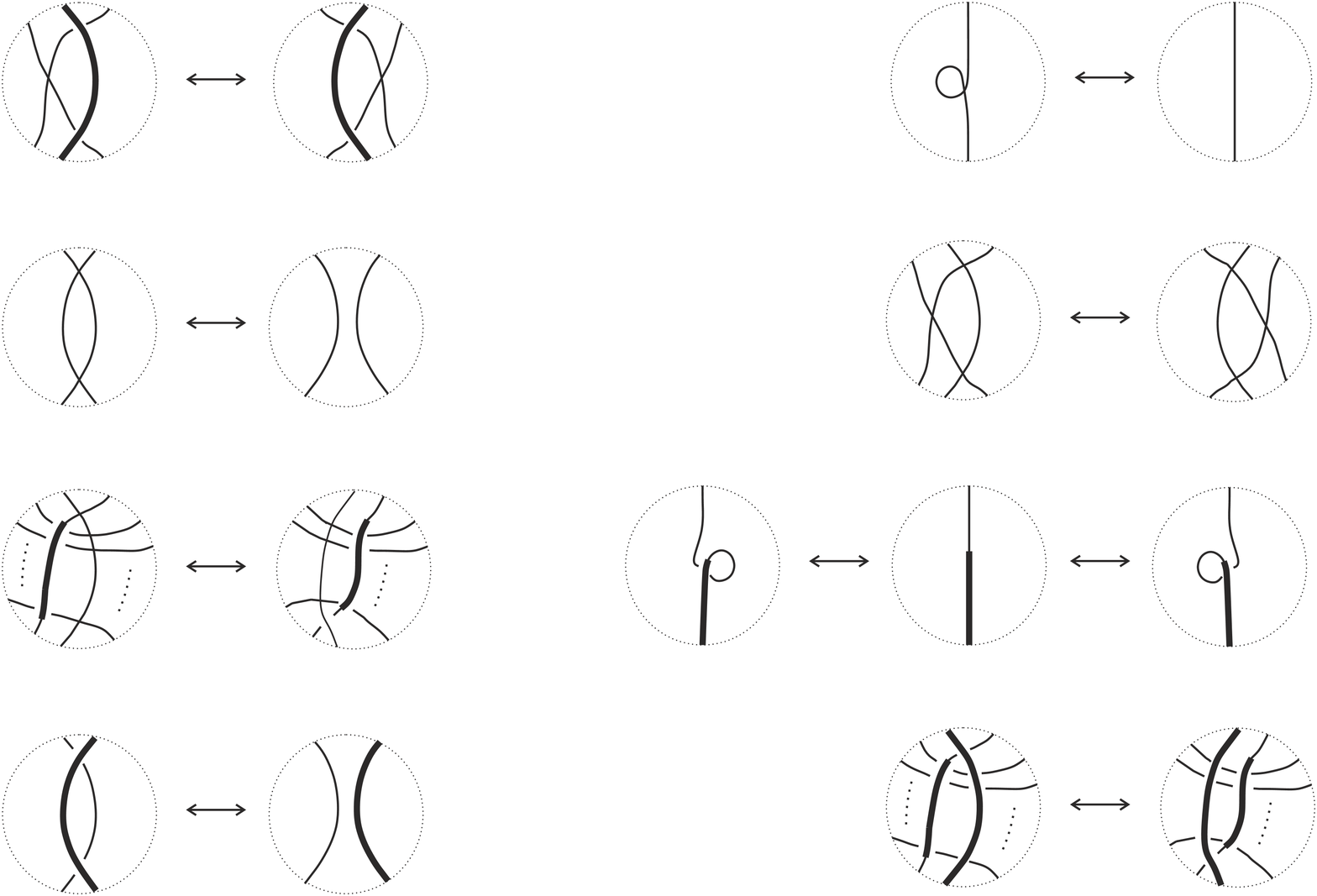}\end{minipage}}\rule{0pt}{1.8cm} & $\begin{minipage}{0.4cm}\ \,\includegraphics[height=1.2cm]{vline}\end{minipage}\longleftrightarrow \ \begin{minipage}{0.4cm}\includegraphics[height=1.2cm]{Reidemeister_agent}\end{minipage}$\\[1.8cm]
\hline
Gau{\ss} diagram, & \ref{SS:Gauss} & \raisebox{-0.5\height}{\includegraphics[width=12pt]{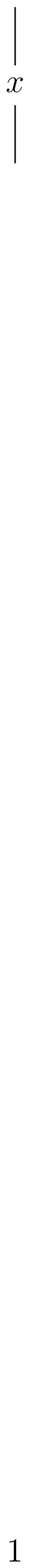}} & \scalebox{0.3}{\begin{minipage}{0.95\textwidth}
\centering
\includegraphics[width=\textwidth]{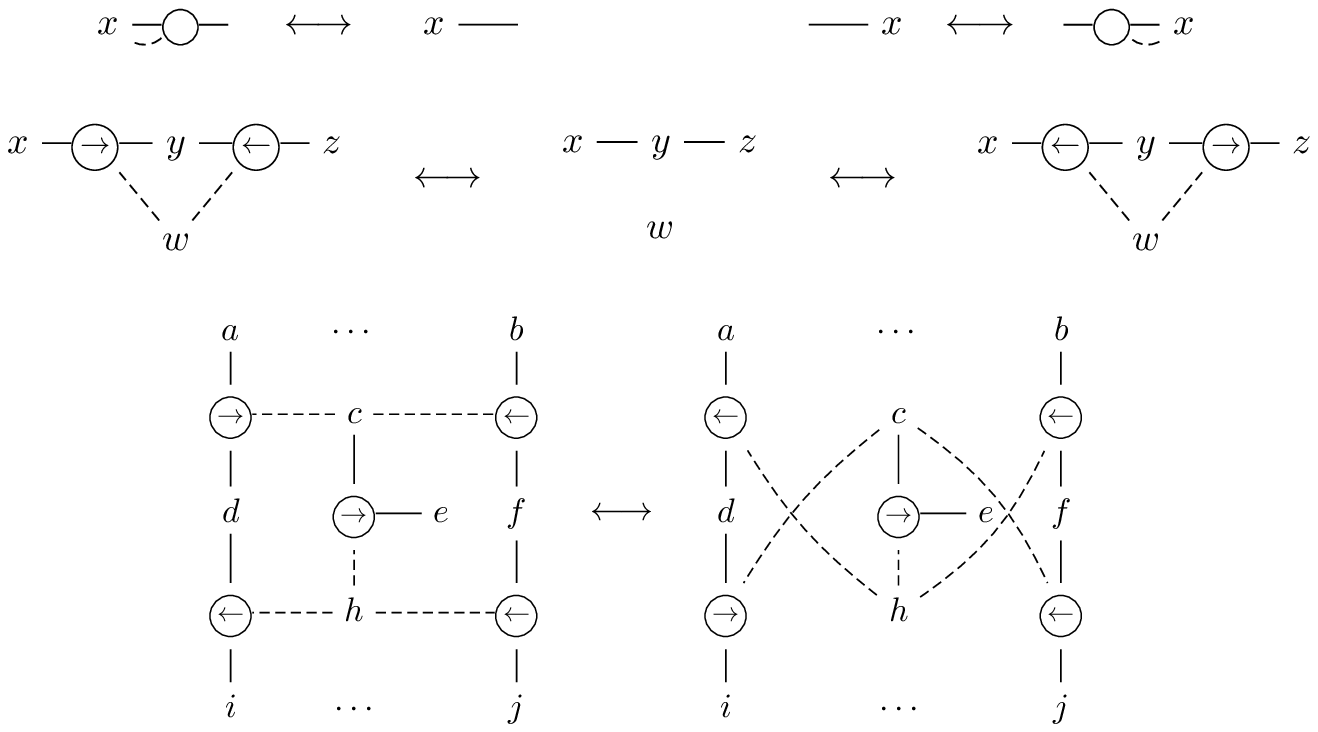}\end{minipage}}\rule{0pt}{1.8cm} & \raisebox{-0.5\height}{\includegraphics[width=0.128\textwidth]{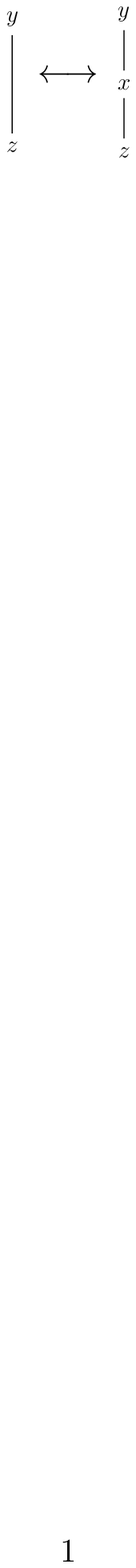}}\\[1.8cm]
\hline
\end{tabular}
\end{center}
\caption{\label{T:diagrams}The five diagrammatic formalisms for Inca foams.}
\end{table}

\subsection{Roseman diagrams of foams}\label{SS:Roseman}

Any embedded surface $F$ in $\mathds{R}^4$ can be drawn in $\mathds{R}^3$ by projecting $F$ onto a choice of $3$--plane $H\subset \mathds{R}^4$. We choose a generic projection so that neighbourhoods of singular points are as shown in Figure~\ref{F:BranchPoints}. Break the surface to keep track of `under' and `over' information. The resulting diagram is called a \emph{broken surface diagram} \cite{CarterKamadaSaito:04}.

\begin{figure}
\centering
\includegraphics[width=4.5in]{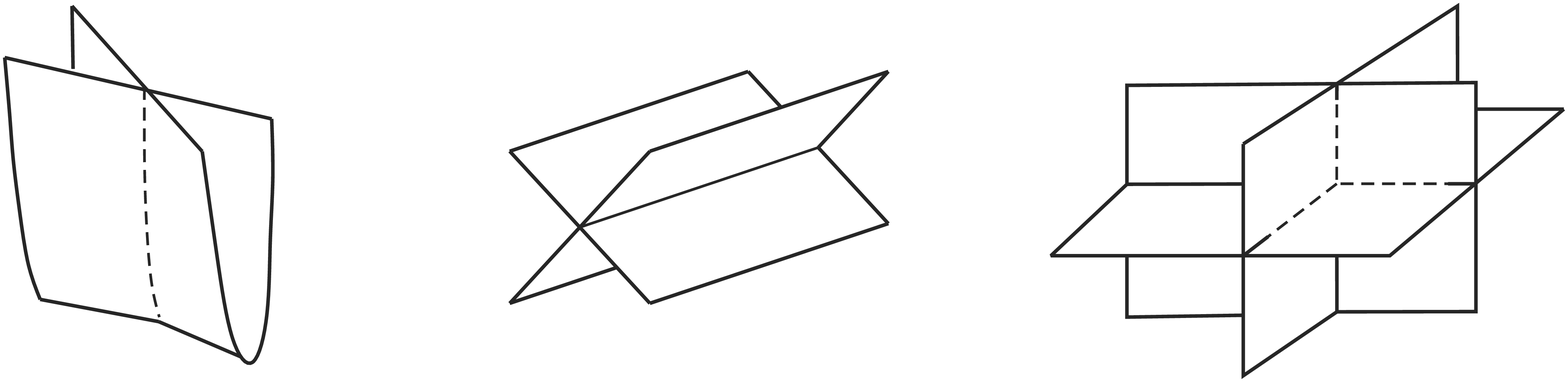}
\caption{\label{F:BranchPoints}Neighbourhoods of singular points of a generic $3$--dimensional projection of a knotted surface in $\mathds{R}^4$.}
\end{figure}

Roseman's Theorem for Foams provides a collection of $15$ local moves on broken surface diagrams so that any two $2$--foams are ambient isotopic if and only if any broken surface diagram of one is related to any broken surface diagram of the other by a finite sequence of these \emph{Carter--Roseman moves} \cite{HommaNagase:85,CarterSaito:93,Roseman:98,Carter:12}.

Two Roseman diagrams are (stably) equivalent  if a pair of $2$--foams which they represent are (stably) equivalent.

\subsection{Roseman diagrams of sphere and interval tangles}\label{SS:SITangle}

Our next diagrammatic formalism allows us to ignore the Carter--Roseman moves which we do not need in our context because our disks are disjoint and our spheres have no local knotting. By convention, when we say \emph{Roseman diagram} without further specification, what we meet is a Roseman diagram of a sphere and interval tangle as defined in this section.

\begin{definition}[Sphere and Interval Tangle]
A \emph{connected sphere and interval tangle} is a union \[L\ass L_1\sqcup L_2\sqcup \cdots\sqcup L_\nu,\] of disjointly embedded objects in standard Euclidean $\mathds{R}^4$ defined as follows:
\begin{itemize}
\item A set $S_1,S_2,\ldots, S_k$ of $2$--spheres embedded in $\mathds{R}^4$ such that there exist disjointly embedded closed $2$--balls $B_1,B_2,\ldots B_k$ with $\partial B_j = S_j$ for $j=1,2,\ldots k$.
\item Identifying $S^4\simeq \mathds{R}^4\cup \set{\infty}$, a set of closed intervals $I_1,I_2,\ldots, I_m$ disjointly embedded in $S^4$ such that each interval endpoint lies on a sphere. We allow no other intersections between intervals and spheres. Write $L_i \ass \bigcup_{j=1}^k I_j\cup S_j$. Only $I_m$ may pass through the point $\infty$, and if it does then we call $L_i$ \emph{open} (because $I_m$ splits into two rays when we restrict to $\mathds{R}^4$) otherwise we call it \emph{closed}.
\end{itemize}
A \emph{sphere and interval tangle} is a union of connected sphere and interval tangles which may intersect one another only at the point $\infty$.
\end{definition}

\emph{Stabilization} of a sphere and interval tangle is:

\begin{equation}
\begin{minipage}{0.6cm}\ \,\includegraphics[height=1.8cm]{vline}\end{minipage}\longleftrightarrow\quad  \begin{minipage}{0.6cm}\includegraphics[height=2cm]{si_agent}\end{minipage}
\end{equation}

Sphere and interval tangles also admit Roseman diagrams. Their equivalence is defined as follows.

\begin{definition}[Equivalence of Roseman diagrams of sphere and interval tangles]
Two Roseman diagrams of sphere and interval tangles are \emph{equivalent} if they are related by a finite sequence of the local moves of Figure~\ref{F:local_roseman}. They are \emph{stably equivalence} if they are related by a finite sequence of these moves and (de)stabilizations.
\end{definition}


\begin{figure}
\centering
\includegraphics[width=\textwidth]{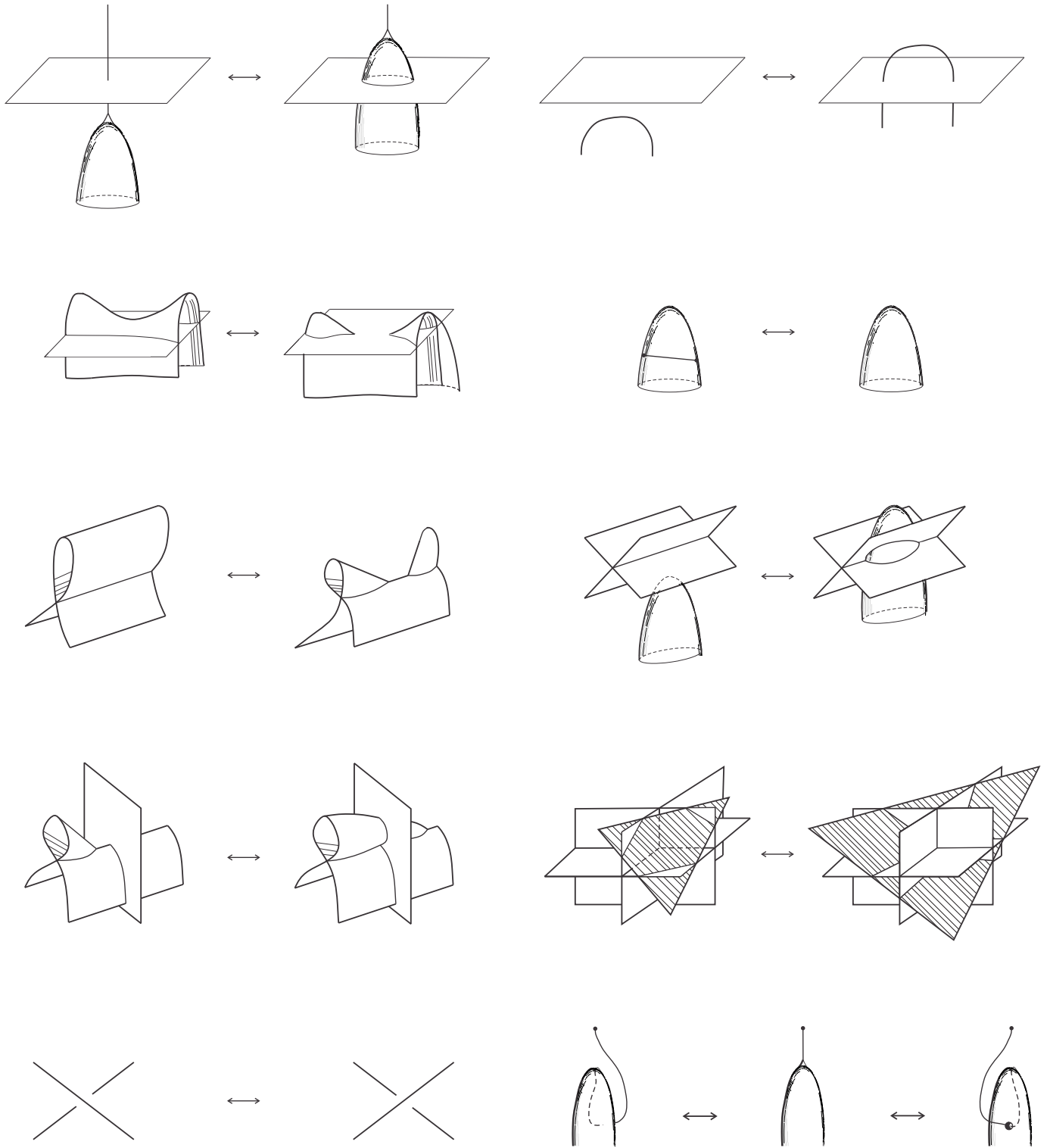}
\caption{\label{F:local_roseman}Local moves on Roseman diagrams of sphere and interval tangles.}
\end{figure}

\subsection{Rosemeister diagrams}\label{SS:Rosemeister}

The interior of the sphere in a Roseman diagram of a sphere and interval tangle plays no role except to confuse. If spheres of Roseman diagrams do not intersect, then we may crush the sphere to a disk without loss of information. One advantage of eliminating redundant sphere interiors is that intervals of Rosemeister diagrams can be coloured, with colours changing as they pass through disks- see Section~\ref{SS:UniversalRack}.

\emph{Stabilization} of a Rosemeister diagram is

\begin{equation}
\begin{minipage}{0.6cm}\ \,\includegraphics[height=1.8cm]{vline}\end{minipage}\longleftrightarrow\quad  \begin{minipage}{0.6cm}\includegraphics[height=2cm]{rosemeister_agent}\end{minipage}
\end{equation}

\begin{definition}[Equivalence of Rosemeister diagrams]
Two Rosemeister diagrams are \emph{equivalent} if they are related by a finite sequence of the local moves of Figure~\ref{F:moves4d2}. They are \emph{stably equivalence} if they are related by a finite sequence of these moves and (de)stabilizations.
\end{definition}

\begin{figure}[htb]
\psfrag{r}[c]{\small\emph{R1}}\psfrag{s}[c]{\small\emph{R2}}\psfrag{t}[c]{\small\emph{R3}}\psfrag{u}[c]{\small\emph{VR}}
\centering
\includegraphics[width=0.95\textwidth]{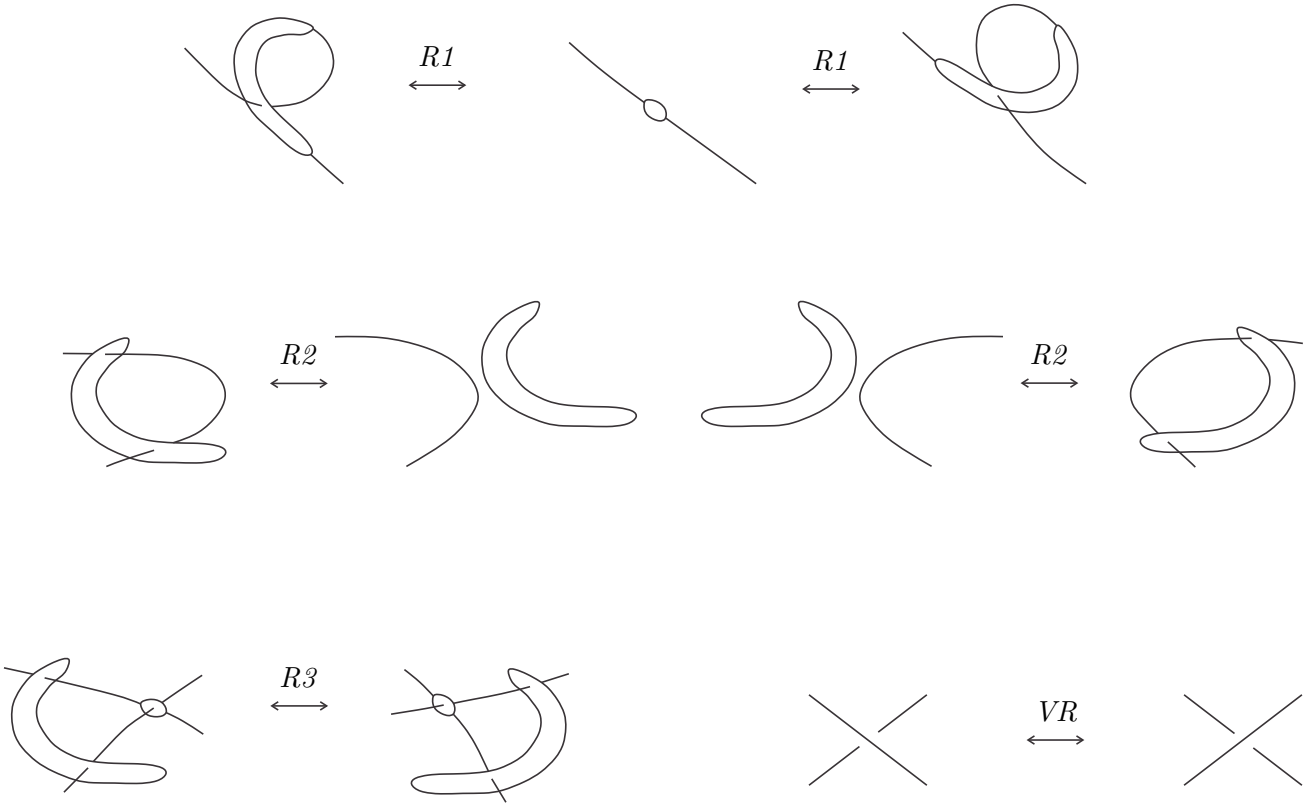}
\caption{\label{F:moves4d2} Local moves on Rosemeister diagrams of sphere-and-interval tangles.}
\end{figure}

\subsection{Reidemeister diagrams}\label{SS:Tangle}

A Reidemeister diagram of a sphere and interval tangle $T$ is a generic projection of $T$ onto a $2$--plane for which images of spheres are disjoint and are designated by thick lines. The authors find this the simplest diagrammatic formalism with which to visualize objects. Stabilization is defined as follows:

\begin{equation}
\begin{minipage}{1cm}\ \,\includegraphics[height=3cm]{vline}\end{minipage}\  \longleftrightarrow \qquad \begin{minipage}{1cm}\includegraphics[height=3cm]{Reidemeister_agent}\end{minipage}
\end{equation}

\begin{definition}[Equivalence of Rosemeister diagrams]
Two Reidemeister diagrams are \emph{equivalent} if they are related by a finite sequence of the local moves of Figure~\ref{F:local_moves_machines}, called \emph{cosmetic moves} and Figure~\ref{F:local_moves_machines1}, called \emph{Reidemeister moves}. They are \emph{stably equivalence} if they are related by a finite sequence of these moves and (de)stabilizations.
\end{definition}

\begin{figure}[htb]
\centering
\psfrag{a}[c]{$\trr$}
\psfrag{b}[c]{$\rrt$}
\psfrag{V}[c]{\small  \emph{I1}}
\psfrag{x}[c]{$x$}
\psfrag{T}[c]{\small \emph{VR1}}\psfrag{R}[c]{\small \emph{VR2}}\psfrag{S}[c]{\small \emph{VR3}}
\psfrag{Q}[c]{\small \emph{SV}}\psfrag{D}[c]{\small \emph{I2}}\psfrag{E}[c]{\small \emph{FM1}}\psfrag{F}[c]{\small \emph{FM2}}\psfrag{C}[c]{\small \emph{I3}}\psfrag{Y}[c]{\small \emph{ST}}
\psfrag{X}[c]{\small \emph{ST}}
\includegraphics[width=0.85\textwidth]{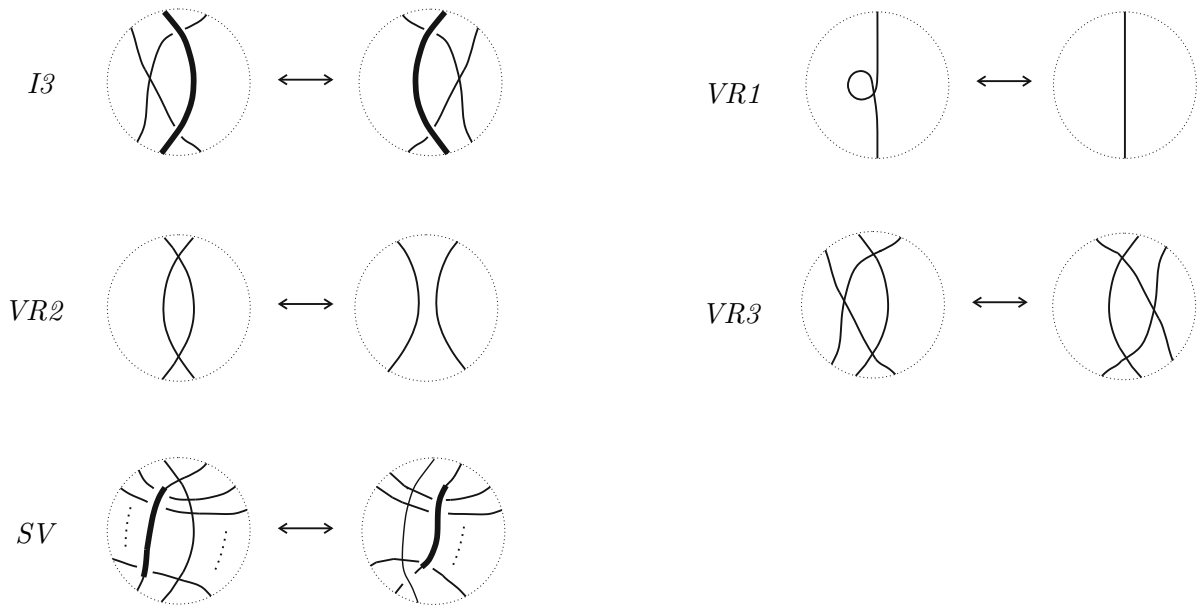}
\caption{\label{F:local_moves_machines} \small Cosmetic moves for Reidemeister diagrams. Directions are not indicated, which means that the moves are valid for all directions so long as the directions on the RHS and on the LHS match up.}
\end{figure}

\begin{figure}[htb]
\centering
\psfrag{T}[c]{\small \emph{VR1}}\psfrag{R}[c]{\small \emph{VR2}}\psfrag{S}[c]{\small \emph{VR3}}
\psfrag{Q}[c]{\small \emph{SV}}\psfrag{D}[c]{\small \emph{R1}}\psfrag{A}[c]{\small \emph{R2}}\psfrag{B}[c]{\small \emph{R3}}\psfrag{C}[c]{\small \emph{UC}}
\psfrag{X}[c]{\small \emph{ST}}
\includegraphics[width=0.85\textwidth]{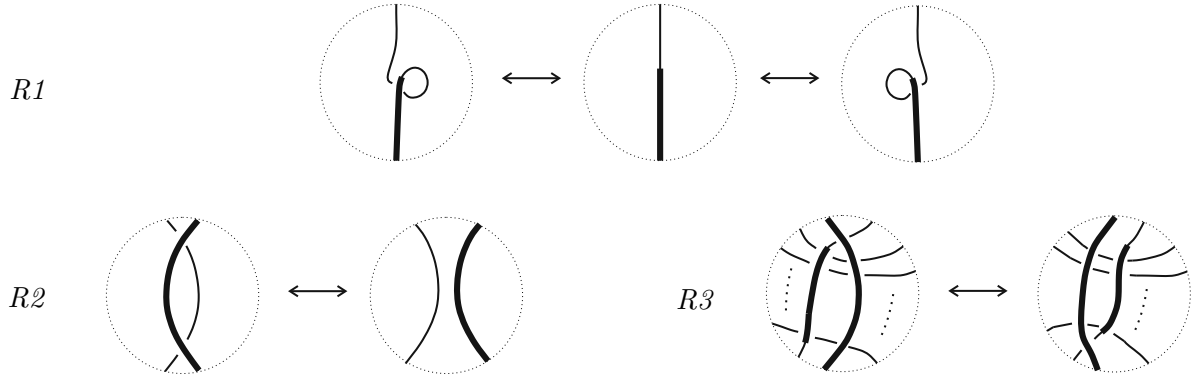}
\caption{\label{F:local_moves_machines1} \small Reidemeister moves for Reidemeister diagrams, valid for any directions of the thick lines (directions on the LHS and on the RHS must match up).}
\end{figure}

\subsection{Gau{\ss} Diagrams}\label{SS:Gauss}

Our final diagrammatic formalism is combinatorial and is based on labeled graphs. It is minimal and as such it's the simplest to use for defining some invariants.

\begin{definition}[Gau{\ss} diagram of an Inca foam]\label{D:TangleMachine}
A \emph{Gau{\ss} diagram of an Inca foam} is a triple $M\ass (G,S,\phi)$ consisting of:
\begin{itemize}
\item A finite graph $G$ that is a disjoint union of path graphs $P_1,\ldots, P_k$  and cycles $C_1,\ldots,C_l$:
\begin{equation}
G\,\ass\, \left(P_1\dU P_2\dU \cdots\dU P_k\right)\dU \left(C_1\dU C_2\dU\cdots\dU C_l\right),
\end{equation}
The graph $G$ is called the \emph{underlying graph} of $M$. 
\item A subset of registers $S\subseteq V(G)$ called \emph{agents}.
\item A multivalued \emph{interaction function} $\phi\colon\, S\Rightarrow E(G)\times \set{\leftarrow,\rightarrow}$ specifying the edges acted on by each agent and a direction $\leftarrow$ or $\rightarrow$.
\end{itemize}
\end{definition}

Two Gau{\ss} diagrams $M_1$ and $M_2$ are considered \emph{equivalent} if they are related by a finite sequence of the following \emph{Reidemeister moves}:

\subsubsection*{Reidemeister \textrm{I}} Here, directions of edges $\leftarrow$ or $\rightarrow$ are arbitrary:
 \begin{equation}
\centering
\raisebox{-0.5\height}{\includegraphics[width=0.75\textwidth]{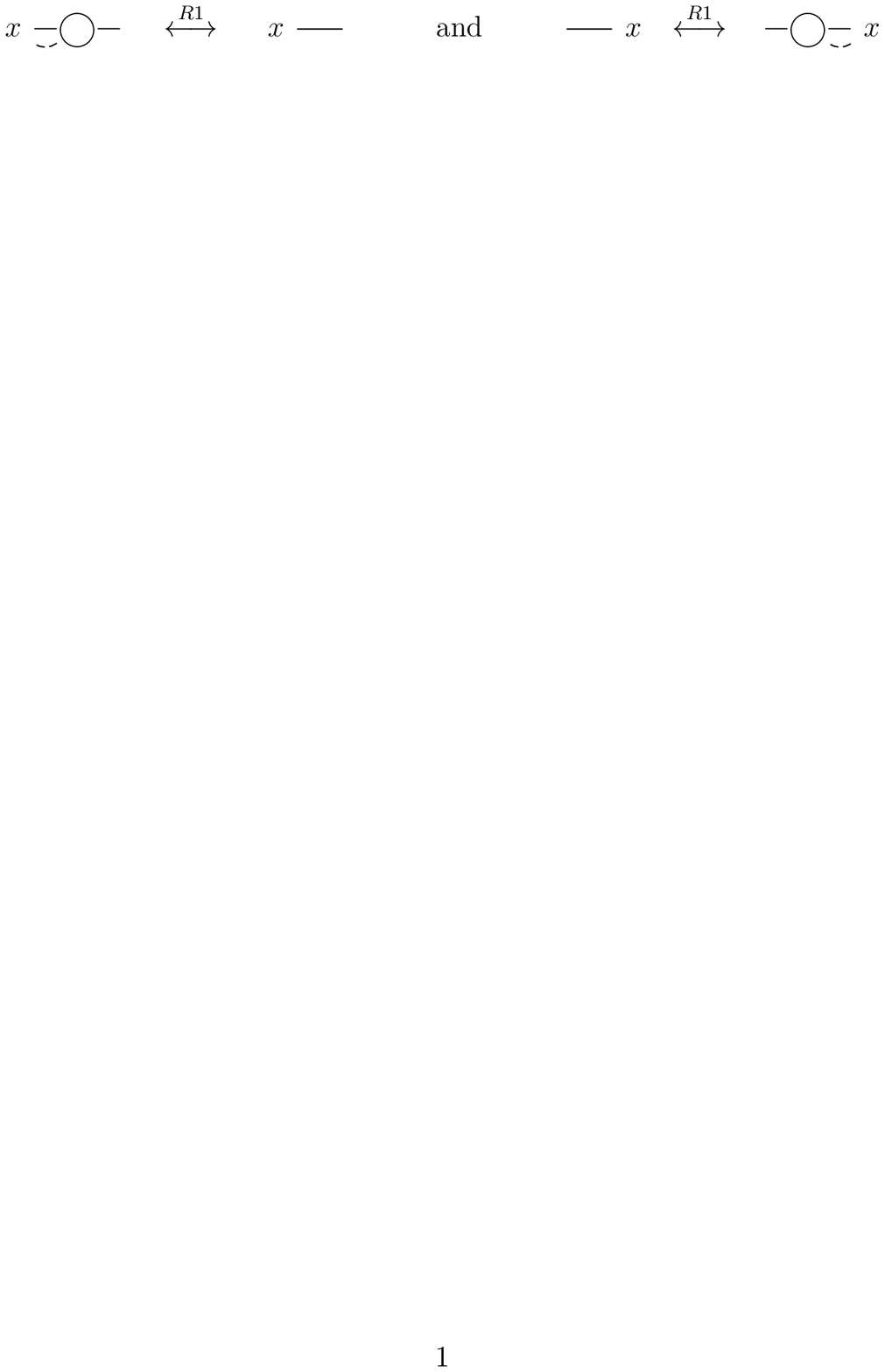}}
\end{equation}

\subsubsection*{Reidemeister \textrm{II}} In following local modification, the top central vertex must be outside the set of agents $S$.
\begin{equation}
\centering
\raisebox{-0.5\height}{\includegraphics[width=0.9\textwidth]{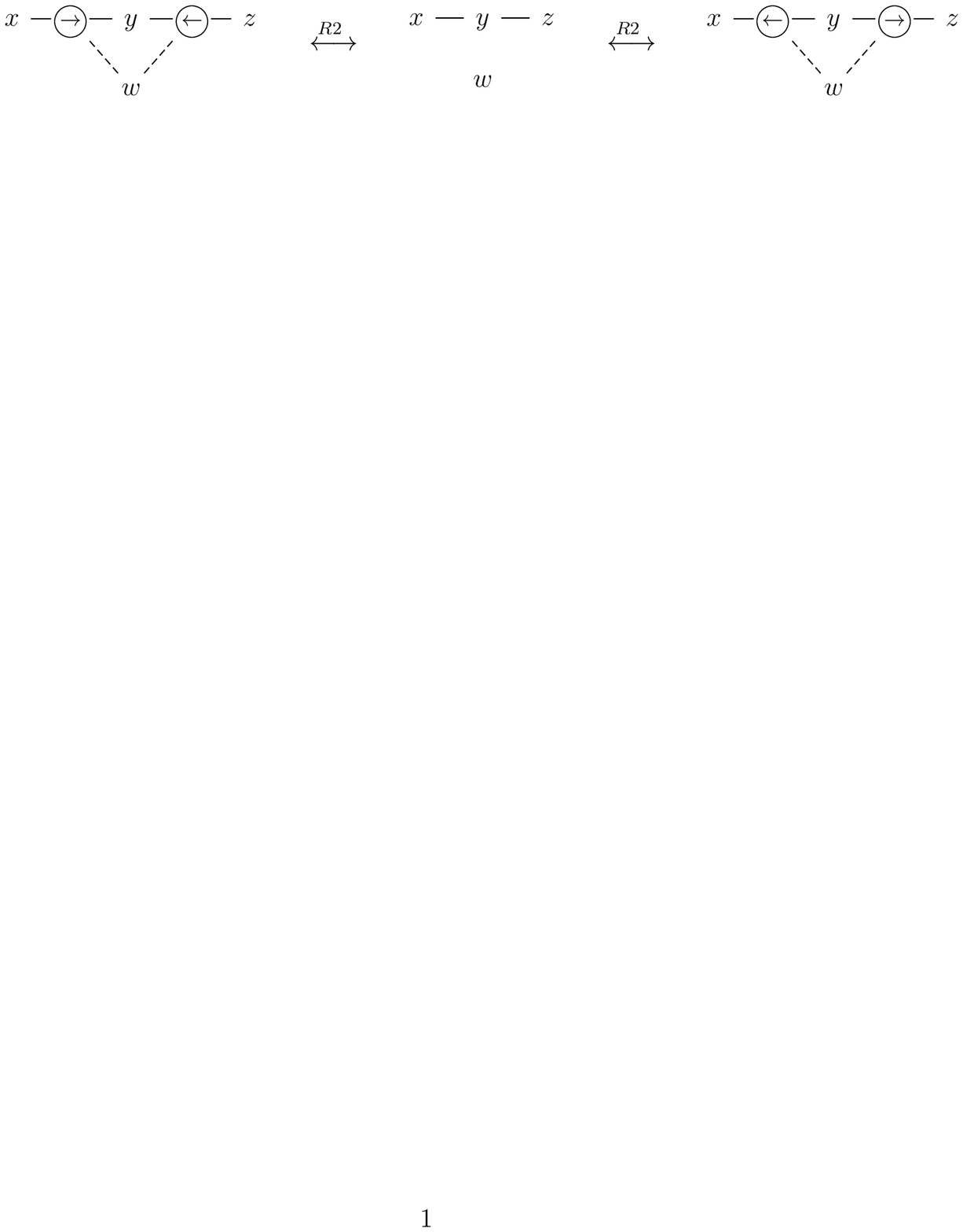}}
\end{equation}

\subsubsection*{Reidemeister \textrm{III}} All edges in $\phi(e)$ in the expression below must participate in the move (the move is invalid for a strict subset of them). Directions $\leftarrow$ or $\rightarrow$ are arbitrary but should correspond on the left and right as indicated in the example below:
\begin{equation}
\centering
\raisebox{-0.5\height}{\includegraphics[width=0.6\textwidth]{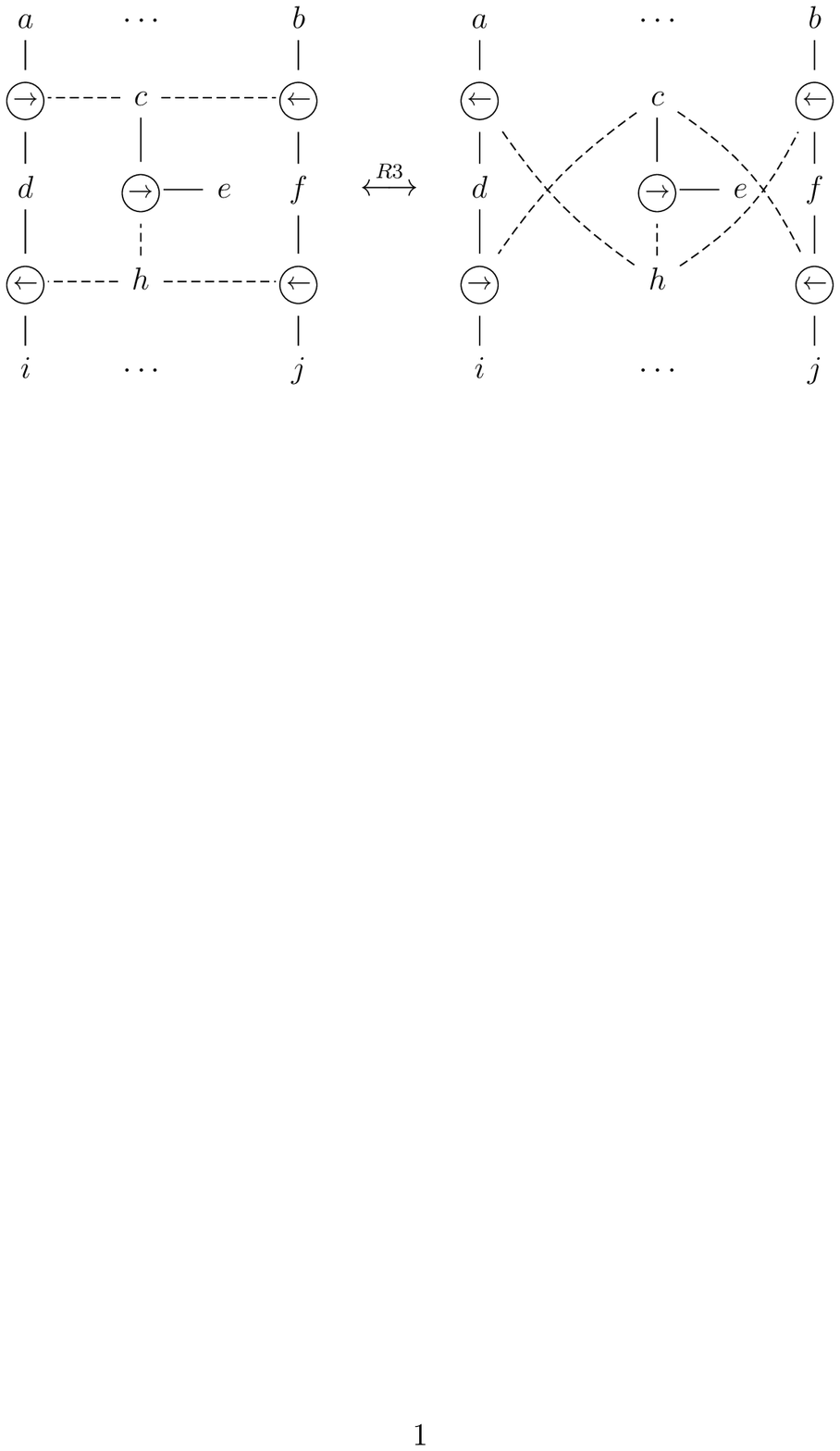}}
\end{equation}

The following move is called \emph{stabilization}, where one of the registers on the LHS must lie outside the image of $\phi$:

\begin{equation}
\centering
\raisebox{-0.5\height}{\includegraphics[width=0.27\textwidth]{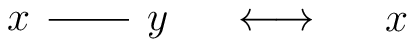}}
\end{equation}

By convention, the stabilization of a single vertex is a $2$--vertex line graph.



\begin{definition}[Equivalence of Gau{\ss} diagrams]
Two Gau{\ss} diagrams are \emph{equivalent} if they are related by a finite sequence of \emph{Reidemeister moves}. They are \emph{stably equivalence} if they are related by a finite sequence of these moves and (de)stabilizations.
\end{definition}

\section{Proof of equivalence}\label{S:FormalismEquivalence}

The goal of this section is the following theorem:

\begin{theorem}\label{T:Equivalence}
Stable equivalence classes of all of the diagrammatic formalisms in Section~\ref{S:FiveDescriptions} are in bijective correspondence with stable equivalence classes of Inca foams.
\end{theorem}

\begin{proof}
\begin{description}
\item[$\text{Gau{\ss} } \Leftrightarrow \text{ Reidemeister}$]

This equivalence was proven in \cite{CarmiMoskovich:15a}. To obtain a Reidemeister diagram from a Gau{\ss} diagram, first destabilize until each edge is in the $\phi$--image of some agent. Then replace interactions as follows:

\begin{equation}\label{E:kebab}
\centering
\raisebox{-0.5\height}{\includegraphics[width=0.8\textwidth]{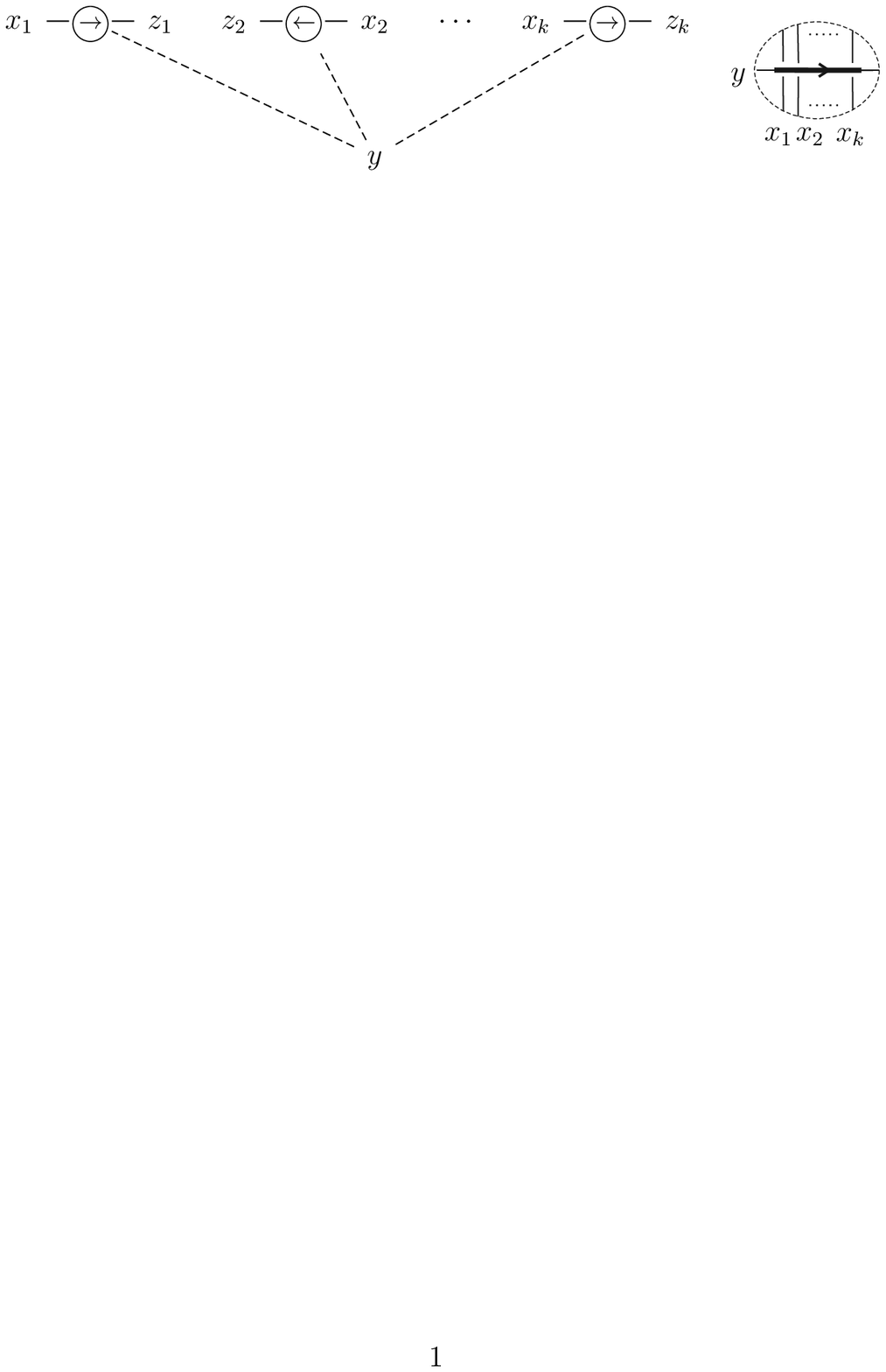}}
\end{equation}

The indeterminacy in the translation from Gau{\ss} diagram interactions to tangle diagram interactions is captured by moves $I1$, $I2$ and $I3$ in Figure~\ref{F:local_moves_machines}.

Then concatenate as dictated by the graph, as in Figure~\ref{F:linking}. The indeterminacy in doing this is captured by moves $V\!R1$, $V\!R2$, $V\!R3$, and $SV$ in Figure~\ref{F:local_moves_machines}. Once tangle endpoints have been `sent to infinity', there are no further indeterminacies.

Reidemeister moves on Gau{\ss} diagrams correspond to Reidemeister moves on Reidemeister diagrams by construction.

\begin{figure}
\includegraphics[width=0.8\textwidth]{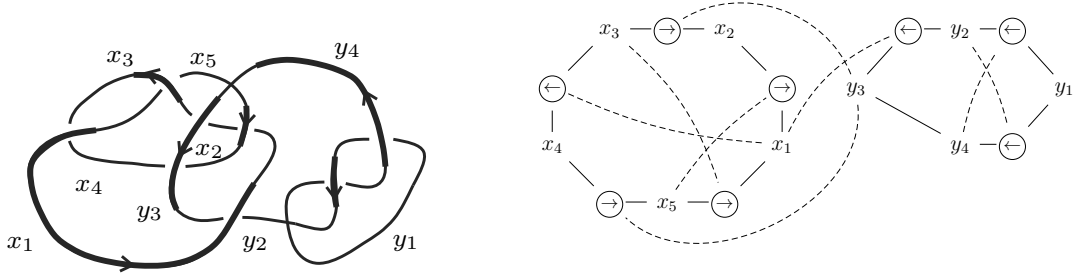}
\caption{\label{F:linking}A Gau{\ss} diagram and corresponding Reidemeister diagram.}
\end{figure}

\item[$\text{Reidemeister } \Leftrightarrow \text{ Rosemeister}$]

Begin by constructing a local model for a single interaction, consisting of a single over-strand $A$ with $k$ strands passing up through it and $l$ strands passing down through it. Consider a $2$--disc in Euclidean $\mathds{R}^4$:

\begin{equation}
D\ass \left\{\left.\rule{0pt}{11pt}(x,0,z,0)\in \mathds{R^4}\right|\, \sqrt{x^2+z^2}=1\right\}.
\end{equation}

Orient the boundary of $D$ counterclockwise. The disk $D$ represents the over-strand (the agent) $A$. We usually draw $D$ pointy at the ends for aesthetic reasons.

Pass through $D$ parameterized intervals $l_j^t$ with $t\in[-2,2]$ so that:
\begin{equation}
l_j^t \ass \left\{
             \begin{array}{ll}
               (\frac{j+1}{l+k+2},t,0,1), & \hbox{for $0<j\leq k$;} \\
               (\frac{j+1}{l+k+2},-t,0,1)\rule{0pt}{12pt}, & \hbox{for $k<j\leq l+k$.}
             \end{array}
           \right.
\end{equation}
Thus, an under-strand passing ``up'' through $A$ corresponds to an interval passing up through $D$, and vice versa. Finally, adjoin two parameterized intervals $l_A^+\ass (0,0,1+t,0)$ and  $l_A^-\ass (0,0,-2+t,0)$ of length $1$.

Concatenate as dictated by the graph. At this point, the $4$--dimensional figure that we have constructed, which consists of disks $D_1,D_2,\ldots, D_N$ and of intervals, lies inside a collection of $4\times 4\times 4\times 4$ cubes $B_1,B_2,\ldots,B_N$. We index these so that $D_i$ lies inside $B_i$ for all $1\leq i\leq N$., and embed the cubes disjointedly in $\mathds{R}^4$. Concatenate by connecting endpoints of intervals on the boundaries of the cubes (these are endpoints of $l_j$ intervals and of $l_A$ intervals) to one another, corresponding to how the registers which represent them connect with one another in $M$. The embedding should be chosen so that the concatenation of two smooth embedded intervals is again a smooth embedded interval. Line segments added for the purpose of concatenation should lie entirely outside $B_1,B_2,\ldots,B_k$, and should not intersect.

Finally, for each intersection $p$ of one of the intervals $l_j$ or $l_A^\pm$ with the boundary $\partial B$ of a cube $B$, endpoints of $l_j$ intervals or of $l_A$ intervals which have not been used for concatenation embed a ray into $\mathds R^4$ so that its endpoint maps to $p$ and its open end gets sent to $\infty$, requiring again that it not intersect any of the other geometric objects which we have placed.

The indeterminacy in choosing concatenation lines is covered by the move which allows us to pass one interval through another in the $3$--dimensional Rosemeister diagram. Reidemeister moves on Reidemeister diagrams and Reidemeister moves on Rosemeister diagrams correspond.

\item[$\text{Rosemeister}  \Rightarrow \text{ Roseman}$]

To obtain a Roseman diagram from a Reidemeister diagram, replace each disk $D$ by a sphere parameterized as:
\begin{equation}\label{E:etrog}
S\ass \left\{\left.\rule{0pt}{11pt}(\sigma(z)x,\sigma(z)y,z,0)\in \mathds{R^4}\right|\, -1\leq x\leq 1;\ \sqrt{y^2+z^2}=1\right\}.
\end{equation}
\noindent where $\sigma\colon\, [-1,1]\to [0,1]$ is a modified logistic function $\frac{1}{2}+\frac{1}{2}\tanh\left(\tan(\frac{\pi}{2}z)\right)$ for $x\in (-1,1)$ and with $\sigma(\pm 1)\ass 0$. We choose this parameterization so as to make a Roseman diagram into a $3$--dimensional projection of a stratified space in order for smooth ambient isotopy of such objects to be well-defined \cite{GoreskyMacPherson:88}. Thus, an interaction in the resulting Roseman diagram looks as in Figure~\ref{F:kebaby}.

Each local move on a Rosemeister diagram corresponds to a fixed finite sequence of Roseman moves on a Roseman diagram.

\begin{figure}[htb]
\begin{minipage}[t]{1\linewidth}
\psfrag{s}[c]{\small $S$}
\psfrag{a}[c]{\small $l_A^+$}
\psfrag{b}[c]{\small $l_A^-$}
\centering
    \includegraphics[width=0.3\textwidth]{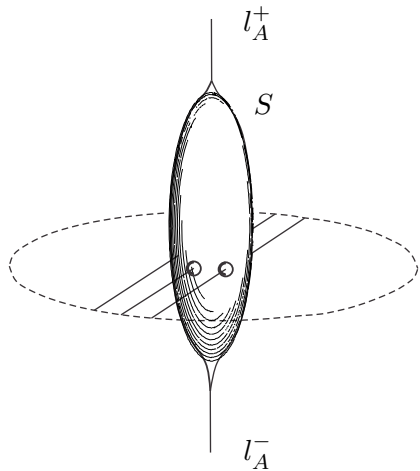}
\end{minipage}
\caption{\label{F:kebaby} A Roseman diagram corresponding to a single interaction.}
\end{figure}

\item[$\text{Roseman}  \Rightarrow \text{ Inca foam}$]

Replace intervals by narrow cylinder with a disc in it. More precisely, replace each interval component $l\colon\, [0,1]\to H$ by the boundary of a embedded cylinder $l_c\colon D^2\times [0,1]\to H$ together with the disc $l_c(D^2\times 1/2)$ as in Figure~\ref{F:construction}. The moves on Roseman diagrams of sphere and interval tangles are restrictions of the set of moves on Roseman diagrams of Inca foams.

\begin{figure}
\centering
\includegraphics[width=0.5\textwidth]{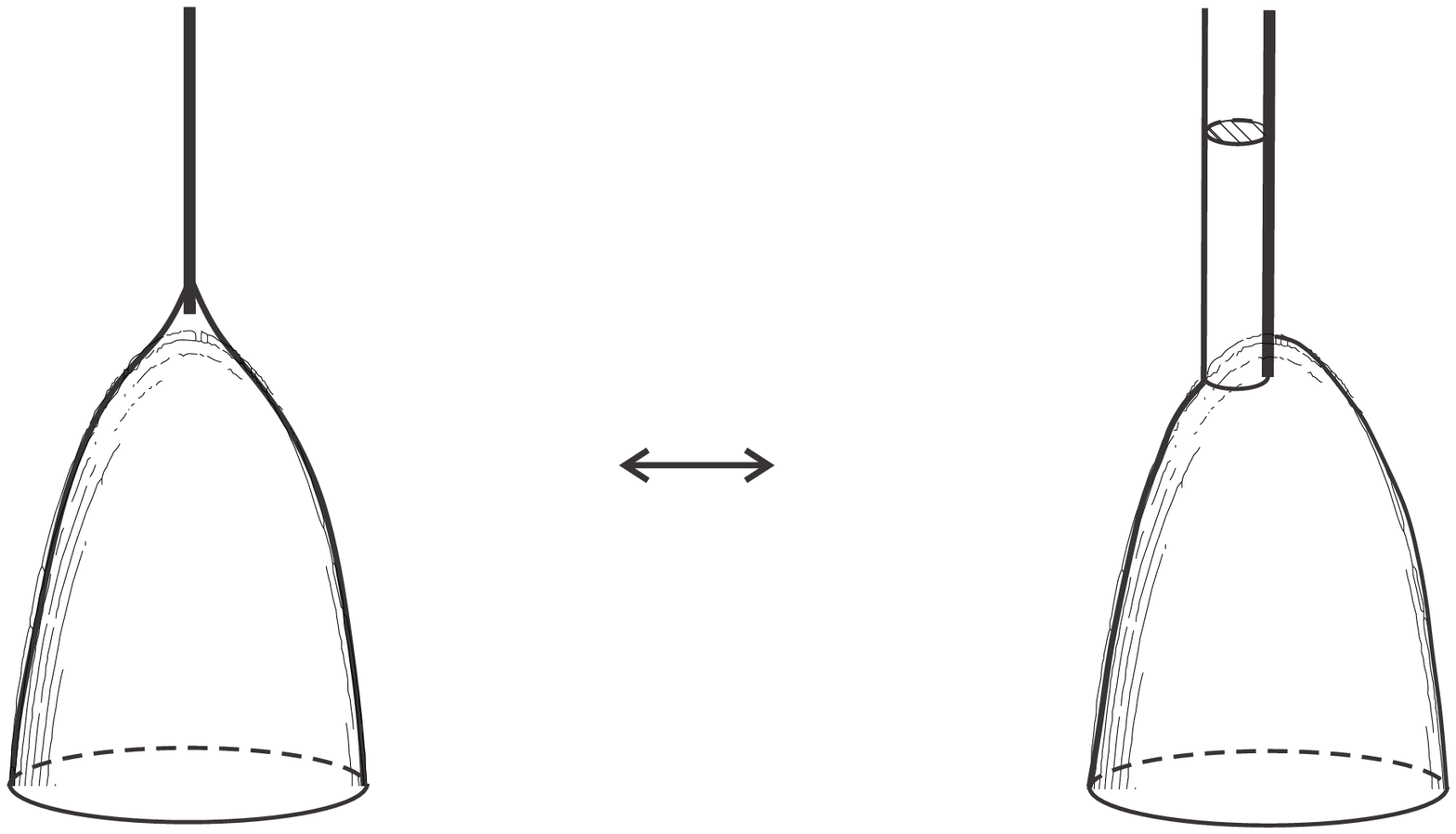}
\caption{\label{F:construction}}
\end{figure}

\item[$\text{Inca foam}  \Rightarrow \text{ Rosemeister}$]

Choose a $0$--cell $p$ inside balls $B$ of the Inca foam $K$. For two balls $B$ and $B^\prime$ which intersect at a disc $D$, join their points $p$ and $p^\prime$ by an embedded $1$--cell in $B\cup B^\prime$ which passes once transversely through $D$. Together, these points and intervals form a $1$--complex $G$ of embedded cycles one of which may pass through $\infty$. The union of balls of $K$ deformation retracts onto a union of intervals and small non-intersecting balls around the points $p_1,p_2,\ldots, p_k$. This retraction may be performed so that at each step we have either an Inca foam or a sphere-and-interval tangle (where an interval might have length $0$).

For concreteness we carry out the contraction via the following procedure. Choose a stratified Morse function $f$ for $\accentset{\circ}{I}\cap B$ \cite{GoreskyMacPherson:88}. By compactness, $\pi B$ contains images of a finite number of critical points of $f$. Inside a small neighbourhood, each critical point is of one of the forms in Figure~\ref{F:cp}.

 \begin{figure}
\centering
\includegraphics[width=0.9\textwidth]{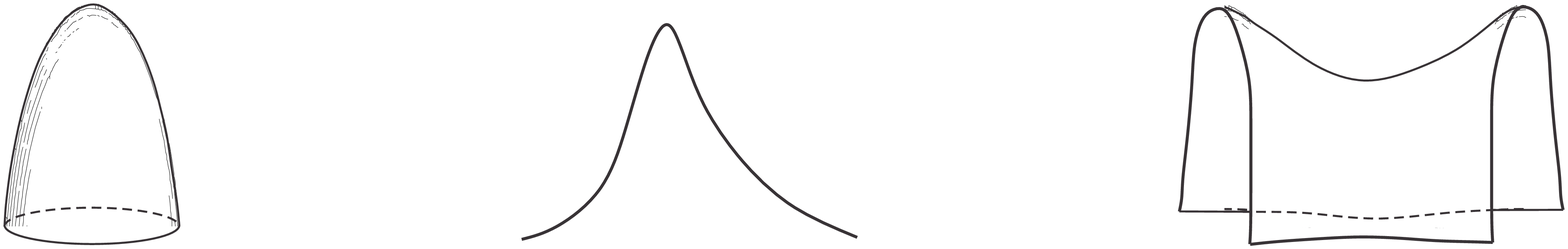}
\caption{\label{F:cp} Possible forms of critical points for a stratified Morse function of a sphere and interval tangle.}
\end{figure}

Projecting onto a hyperplane $H$ and replacing the balls by discs gives a Rosemeister diagram $D$ for $K$, and different choices of $H$ give Roseman diagrams related by Roseman moves by Roseman's Theorem. Similarly, retracting the intersection discs between the balls to points and then extending them into small intervals gives a sphere and interval tangle, and the Roseman moves on a sphere and interval tangle are the restriction of the Roseman moves on a $2$--foam.

It remains to prove that different choices of $G$ give rise to equivalent Rosemeister diagrams. Let $p$ be the point in the projection $\pi$ to $H$ of $B$ with $x$--coordinate $x\in[x_1,x_2]$ in the Roseman diagram. For sufficiently small $\epsilon>0$ there are no critical points of $f$ in $B^1\ass [x-\epsilon,x+\epsilon]\times[y_1,y_2]\times[z_1,z_2]\subseteq \pi(B)$. As we shrink $[x_1,x_2]$ to $[x-\epsilon,x+\epsilon]$, the boundary of $\pi (B)$ will cross over critical points of the image of $f$. By induction and by general position, after shrinkage this ball contains only line segments between the planes $\{x-\epsilon\}\times[y_1,y_2]\times[z_1,z_2]$ and $\{x+\epsilon\}\times[y_1,y_2]\times[z_1,z_2]$ without critical points, and also $2$--dimensional components (parts of boundaries of other balls) without critical points. Next, cut out $\pi (B^1)$, scale it to a ball $B^2$ of radius $\epsilon$ around $p$, and connect endpoints and end-lines on $\pi (\partial B^1)$ to endpoints and end-lines on $\pi (\partial B^2)$ with straight lines and broken surfaces without critical points. For sufficiently small epsilon, there will be no $2$--dimensional components intersecting $\partial B^1$. The embedded object which we obtain is independent of the order by which we shrink the balls (Diamond Lemma). We may now replace the balls by discs. Up to reparametrization this is a Rosemeister diagram.

Inside a $3$--ball, if the critical point is of a $1$--dimensional stratum and if $x\in[x_1,x_2]$ lies below it, then the local move results in a sphere-and-interval tangle whose Reidemeister diagram differs from the original by an R2~move.

 \begin{equation}
\centering
\includegraphics[width=0.6\textwidth]{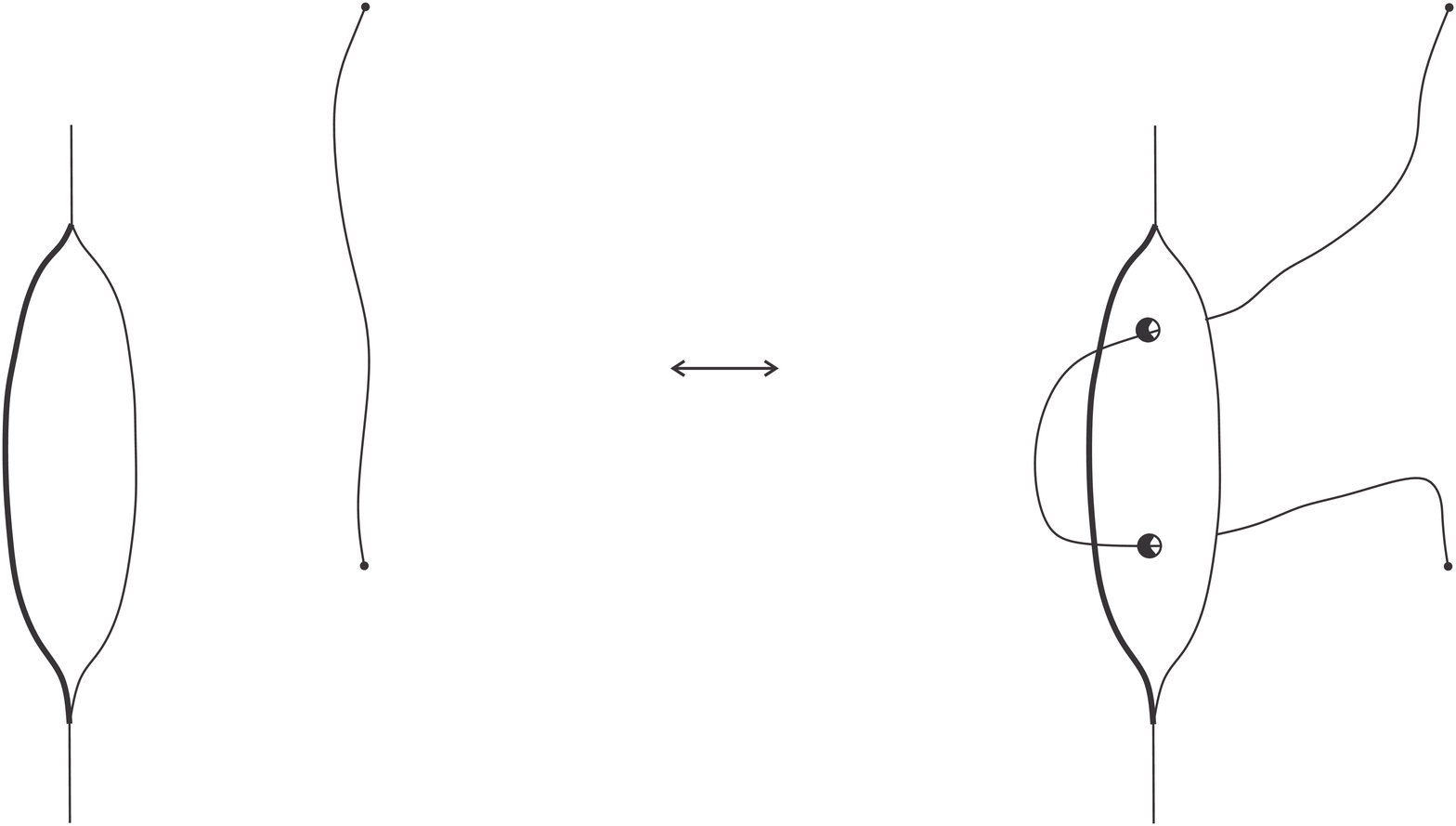}
\end{equation}

If the critical point is of a $2$--dimensional stratum and if $x\in[x_1,x_2]$ lies below it, then the local move results in a sphere-and-interval tangle whose Reidemeister diagram differs from the original by an R3~move.

\begin{equation}\label{E:PushMove}
\psfrag{r}[r]{\small\emph{R3}}
\centering
\includegraphics[width=0.8\textwidth]{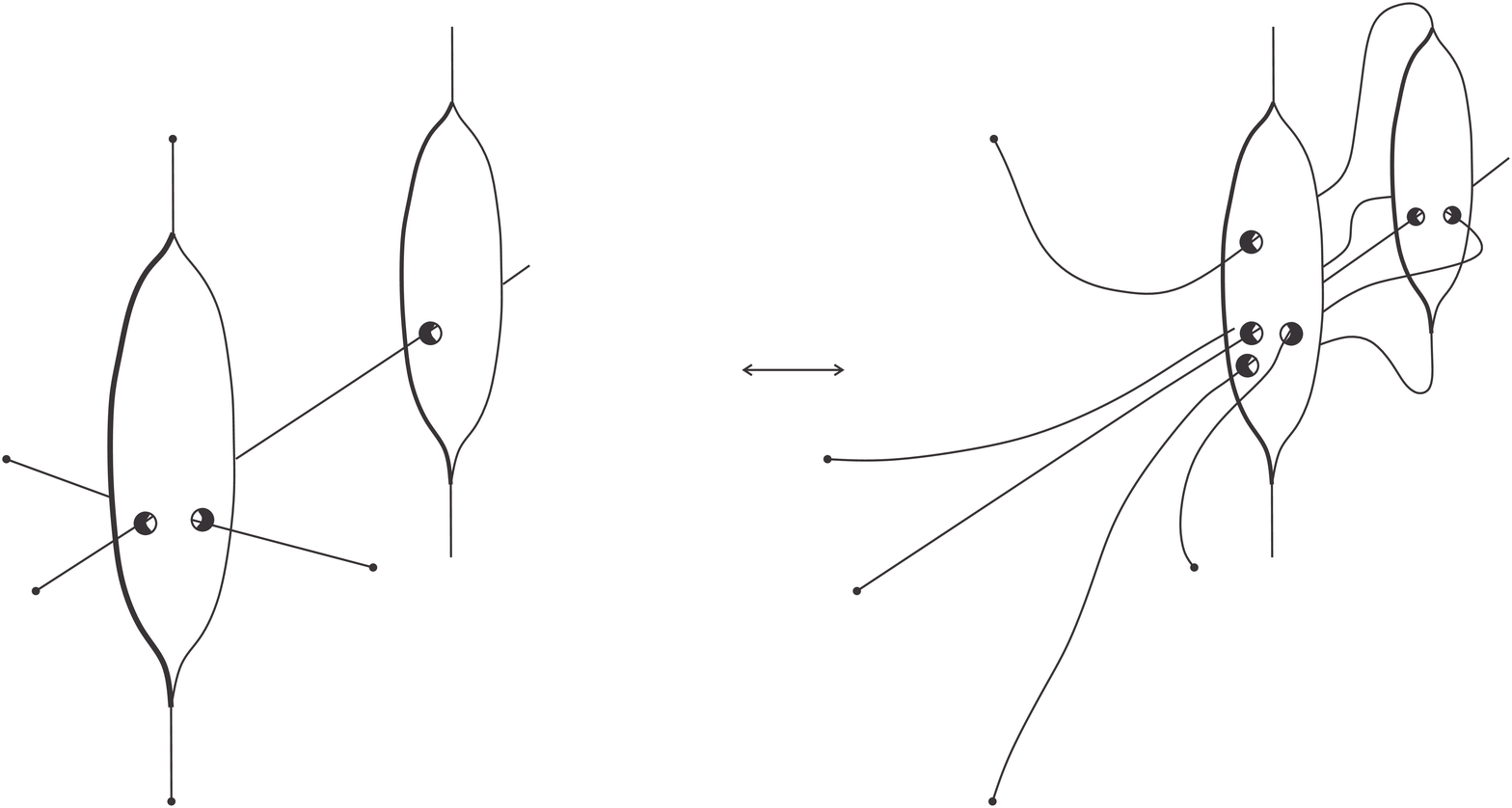}
\end{equation}

An R1~move parallels the corresponding move on Roseman diagrams of sphere and interval tangles. Finally, stabilization of Rosemeister diagrams corresponds to stabilization of Inca foams.
\end{description}
\end{proof}

\section{Invariants}\label{S:Invariants}

In this section we describe some simple characteristic quantities associated to equivalence classes of Inca foams. Such quantities are called \emph{invariants}. An invariant is called \emph{stable} if it is an invariant of stable equivalence classes.

\begin{remark}
Category theory allows a precise definition: Invariants are functors out of a category of Inca foams whose morphisms are equivalences, or out of a closely related category.
\end{remark}

\subsection{Underlying graph}\label{SS:UnderGraph}

Given a Gau{\ss} diagram $M=(G,S,\phi)$, the pair $(G,S)$ is an Inca foam invariant called the \emph{underlying graph}.

Define a vertex $r$ in $G$ to be a \emph{trivial agent} if $M$ is equivalent to a Gau{\ss} diagram $M^\prime$ in which $r$ is not an agent. The number of nontrivial agents is a stable invariant, which can be seen in a corresponding sphere and interval tangle $T$ by choosing a decomposing sphere intersecting the $T$ at $2$ points and containing no spheres in $T$ besides the sphere corresponding to $r$.

The graph $G^{\text{reduced}}$ obtained from $G$ by contracting all trivial agents is a stable Inca foam invariant.

\subsection{Underlying w-knotted object}\label{SSS:wknotsrel}

A \emph{w-tangle} is an algebraic object obtained as a concatenation of {\includegraphics[width=20pt]{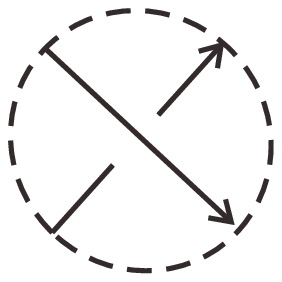}} and {\includegraphics[width=20pt]{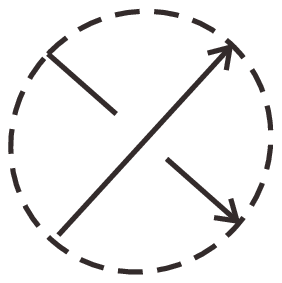}} in the plane. Two w-tangles are \emph{equivalent} if they are related by a finite sequence of Reidemeister moves as shown in Figure~\ref{F:WTangleReidemeister} \cite{FennRimanyiRourke:97,BarNatanDancso:13}.

The diagrammatic calculus of w-knots is similar to the diagrammatic calculus of Reidemeister diagrams, and indeed cutting up w-knots into \emph{w-knotted tangles} has been represented by a \emph{ball and hoop model} which is similar to our sphere-and-interval tangles, although different knotted objects in $4$--space are being described \cite{BarNatan:13}. 

There is no well-defined map from a w-tangle to a sphere-and-interval tangle or vice versa. However, the space of \emph{equivalence classes} of w-tangles is a quotient of the space of \emph{stable equivalence classes} of Reidemeister diagrams by the following \emph{false stabilization} move with no conditions imposed on $x$ and $y$ (if $x,y\in S$ then this move is not stabilization).

\begin{equation}\label{E:FalseStabilization}
\centering
\raisebox{-0.5\height}{\includegraphics[width=0.27\textwidth]{tikzeps5.eps}}
\end{equation}

The difference between equivalence classes of w-tangles and of Reidemeister diagrams of Inca foams lies in how we treat the over-strands. True and false stabilization combine to suppress over-strands, so that Reidemeister moves for Reidemeister diagrams coincide, in the quotient, with Reidemeister moves for w-tangles. We thus have the following:

\begin{proposition}\label{P:TMtoW}
The space of equivalence classes of w-tangles is isomorphic to the quotient of the space of stable equivalence classes of Inca foams by false stabilization.
\end{proposition}

\begin{figure}[htb]
\psfrag{T}[r]{\small \emph{VR1}}
\psfrag{R}[r]{\small \emph{VR2}}
\psfrag{S}[r]{\small \emph{VR3}}
\psfrag{Q}[r]{\small \emph{SV}}
\psfrag{A}[r]{\small \emph{R2}}
\psfrag{B}[r]{\small \emph{R3}}
\psfrag{C}[r]{\small \emph{UC}}
\psfrag{D}[r]{\small \emph{R1}}
\centering
\includegraphics[width=0.9\textwidth]{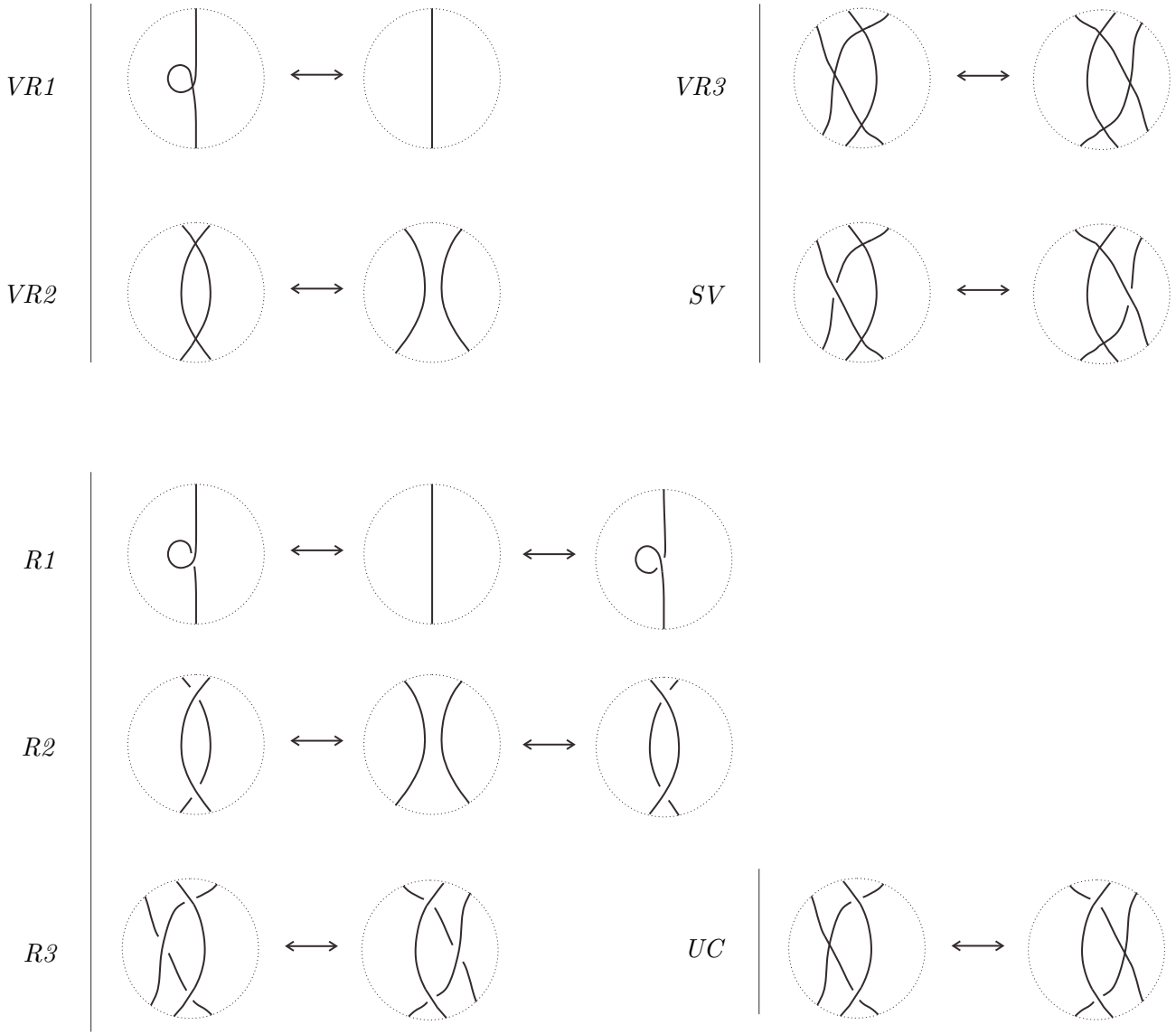}
\caption{\label{F:WTangleReidemeister}Reidemeister moves for w-tangles.}
\end{figure}

\begin{definition}\label{D:UnderlyingWKnot}
If a w-tangle $K$ corresponds to a stable equivalence class of Inca foams to which $M$ belongs, then we say that $K$ is the \emph{underlying w-tangle of $M$}.
\end{definition}

\begin{remark}
Satoh's Conjecture is that equivalence classes of w-knots are in bijective correspondence with a certain class of knotted tori in $\mathds{R}^4$ known as \emph{ribbon torus knots} \cite{Satoh:00}. If a connected Inca foam has $\infty$ outside it, then it is equal to such an embedded knotted torus with discs inside it, whose boundaries are meridians of the torus. The topological explanation of Proposition~\ref{P:TMtoW} is that false (de)stabilization is the operation of adding and taking away such discs (leaving at least one, so in particular Theorem~\ref{T:Equivalence} does imply Satoh's Conjecture).
\end{remark}

\begin{remark}
Like Inca foams, ribbon torus knots objects also admit Rosemeister diagrams. See Figure~\ref{F:WKnot}.
\end{remark}

\begin{figure}
\centering
\includegraphics[width=0.5\textwidth]{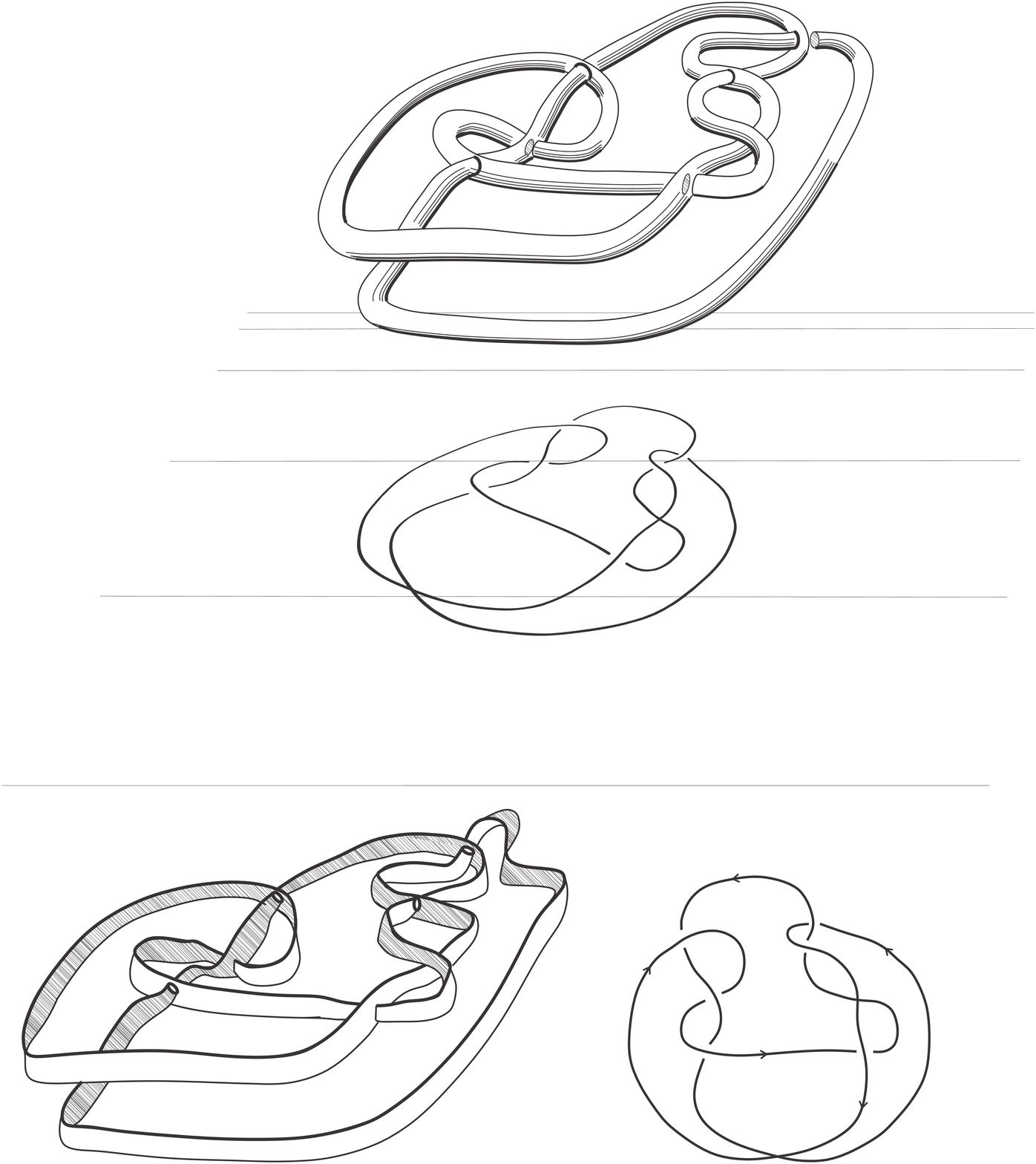}
\caption{\label{F:WKnot} A Reidemeister diagram, a Roseman diagram, and a Rosemeister diagram of a ribbon torus knot.}
\end{figure}

\subsection{Fundamental quandle}\label{SS:UniversalRack}

Define a \emph{quandle} to be a set $Q$ of \emph{colours} equipped with a set $B$ of binary operations subject to the following three axioms:

\subsubsection*{Idempotence:} $x\trr x=x$ for all $x\in Q$ and for all $\trr\in B$.
\subsubsection*{Reversibility:} The map $\trr y\colon\, Q\to Q$, which maps each colour $x\in Q$ to a corresponding colour $x\trr y\in Q$, is a bijection for all $(y,\trr)\in (Q,B)$. In particular, if $x\trr y = z\trr y$ for some $x,y,z\in Q$ and for some $\trr\in B$, then $x=z$.
\subsubsection*{Distributivity:} For all $x,y,z\in Q$ and for all $\trr,\brr \in B$:
\begin{equation}\label{E:Distributivity} (x\trr y)\brr z= (x\brr z)\trr (y\brr z)\enspace .\end{equation}

\begin{remark}
The usual definition of a quandle is the case that $B$ consists of only a single element $\trr$ and its inverse $\rrt$ (\textit{e.g.} \cite{Joyce:82}). Our notion of quandle follows Przytycki \cite{Przytycki:11} who named such a structure a \emph{multi-quandle}.
\end{remark}

A \emph{$(Q,B)$--colouring} $\rho$ of a Gau{\ss} diagram $(G,S,\phi)$ is an assign an element of $B$ to element of $S$ and an element of $Q$ to each vertex of $G$ (in particular, elements of $S$ are coloured both by an element of $B$ and by an element of $Q$). The element of $Q$ by which a vertex is coloured is called its \emph{colour}. The condition that must be satisfied is that for each edge $e\in \phi(v)$, if the colour of the tail of $e$ is $x$, the colour of $v$ is $y$, and the operation of $v$ is $\brr$, then the colour of the head of $e$ must be $x\brr y$. To simplify notation, we write the operation directly onto the Gau{\ss} diagram edge.

\begin{equation}\label{E:QInteract}
\centering
\raisebox{-0.5\height}{\includegraphics[width=0.18\textwidth]{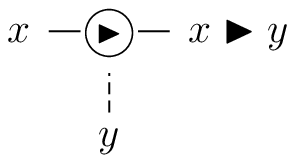}}
\end{equation}

Given a Gau{\ss} diagram $(G,S,\phi)$  colour the vertices of $G$ by distinct formal symbols $c_1,c_2,\ldots,c_N$ and colour the agents in $S$ by distinct elements of $B$, and impose the relations given by \eqref{E:QInteract}. This procedure gives rise to a universal structure $(Q(M),B(M))$ for a Gau{\ss} diagram $M=(G,S,\phi)$, which is a stable invariant. The pair $(Q(M),B(M))$ is called the \emph{fundamental quandle} of $M$.

\begin{remark}
If $B$ is a single-element set then $(Q(M),B(M))$ descends to an invariant of the underlying w-knotted object of $M$.
\end{remark}


\subsection{Linking graph}\label{SS:LinkingGraph}

Consider a Gau{\ss} diagram $M=(G,S,\phi)$, and let $P_1,P_2,\ldots,P_\nu$ denote the connected components of $G$. The \emph{linking number} of vertex $r$ with connected component $j$ is the number of edges $e$ in component $j$ such that $e\in \phi(r)$ and the direction of $e$ is $\rightarrow$, minus the number of edges $e$ in process $j$ such that $e\in \phi(r)$ with direction $\leftarrow$. The \emph{linking graph} $\mathrm{Link}(M)$ of $M$ is a labeling of each vertex in $G$ by a \emph{linking vector} $v^r\ass \left(v^r_{1},v^r_2\ldots,v^r_{\nu}\right)$ whose $k$th entry is the linking number of $r$ with component $k$. The \emph{unframed linking graph} $\mathrm{Link}_0(M)$ is the labeled graph obtained by setting to zero the entry in each linking vector $v^r$ which represents the interactions of $r$ with its own process $P$.

Both the linking graph and the unframed linking graph are Inca foam invariants. This is because an R2 move cancels or creates a pair of inverse interactions $\trr$ and $\rrt$ by the same agent, while an R3 move has no effect on any linking vector, and the effect of an R1 move is only on the `diagonal' entries.

The \emph{reduced linking graph} $\widetilde{\mathrm{Link}}(M)$ of a linking graph $\mathrm{Link}(M)$ is the labeled graph obtained from $\mathrm{Link}(M)$ by contracting all $2$--valent vertices with zero linking vector from the graph (contracting an edge incident to them). The  \emph{reduced unframed linking graph} $\widetilde{\mathrm{Link}_0}(M)$ is defined analogously. These reduced graphs are stable invariants.

\begin{example}
Consider the following Gau{\ss} diagram of which the $j$th vertex in the $i$th component is labeled $x_{ij}$.
\begin{equation}
\centering
\raisebox{-0.5\height}{\includegraphics[width=0.89\textwidth]{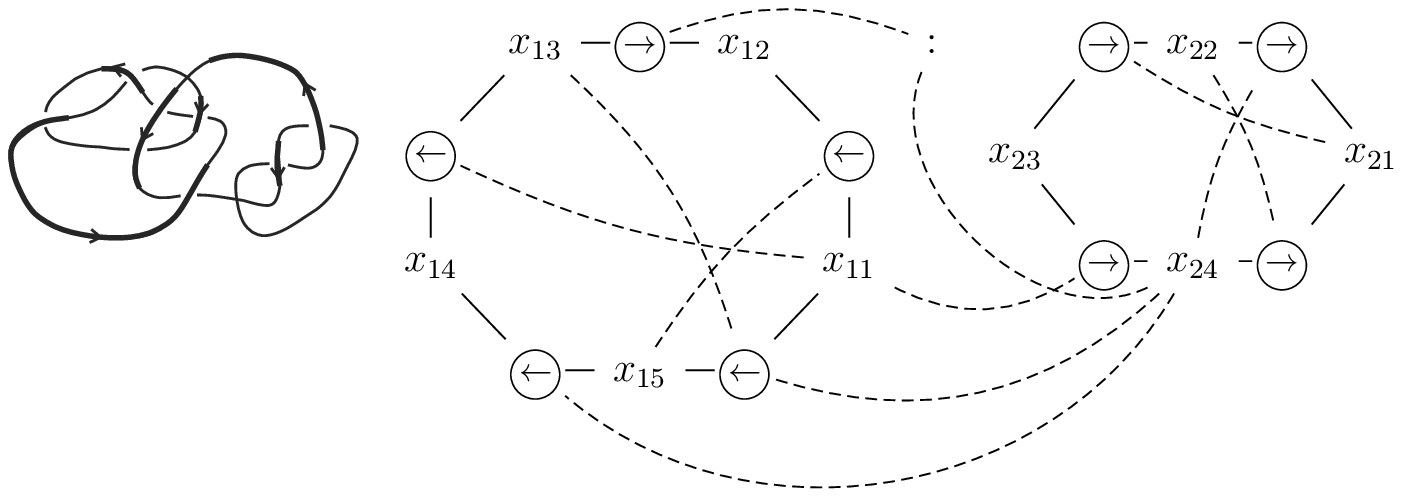}}
\end{equation}
For this Gau{\ss} diagram, the linking graph, $\mathrm{Link}(M)$, and its corresponding stabilization, $\mathrm{Link}_0(M)$, (depicted below using squiggly arrows)
are obtained as
\begin{equation}
\centering
\raisebox{-0.5\height}{\includegraphics[width=0.8\textwidth]{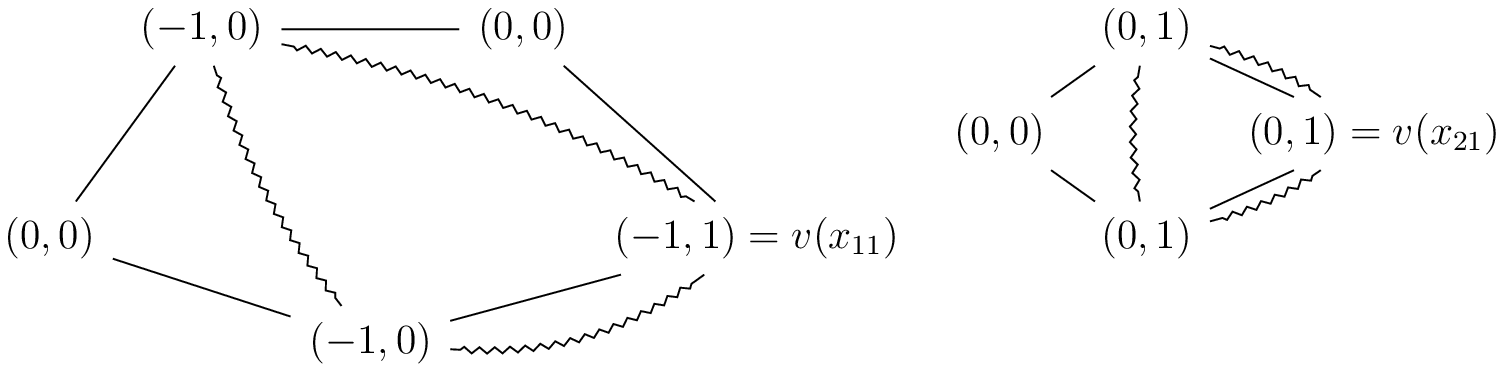}}
\end{equation}
Their unframed counterparts are
\begin{equation}
\centering
\raisebox{-0.5\height}{\includegraphics[width=0.82\textwidth]{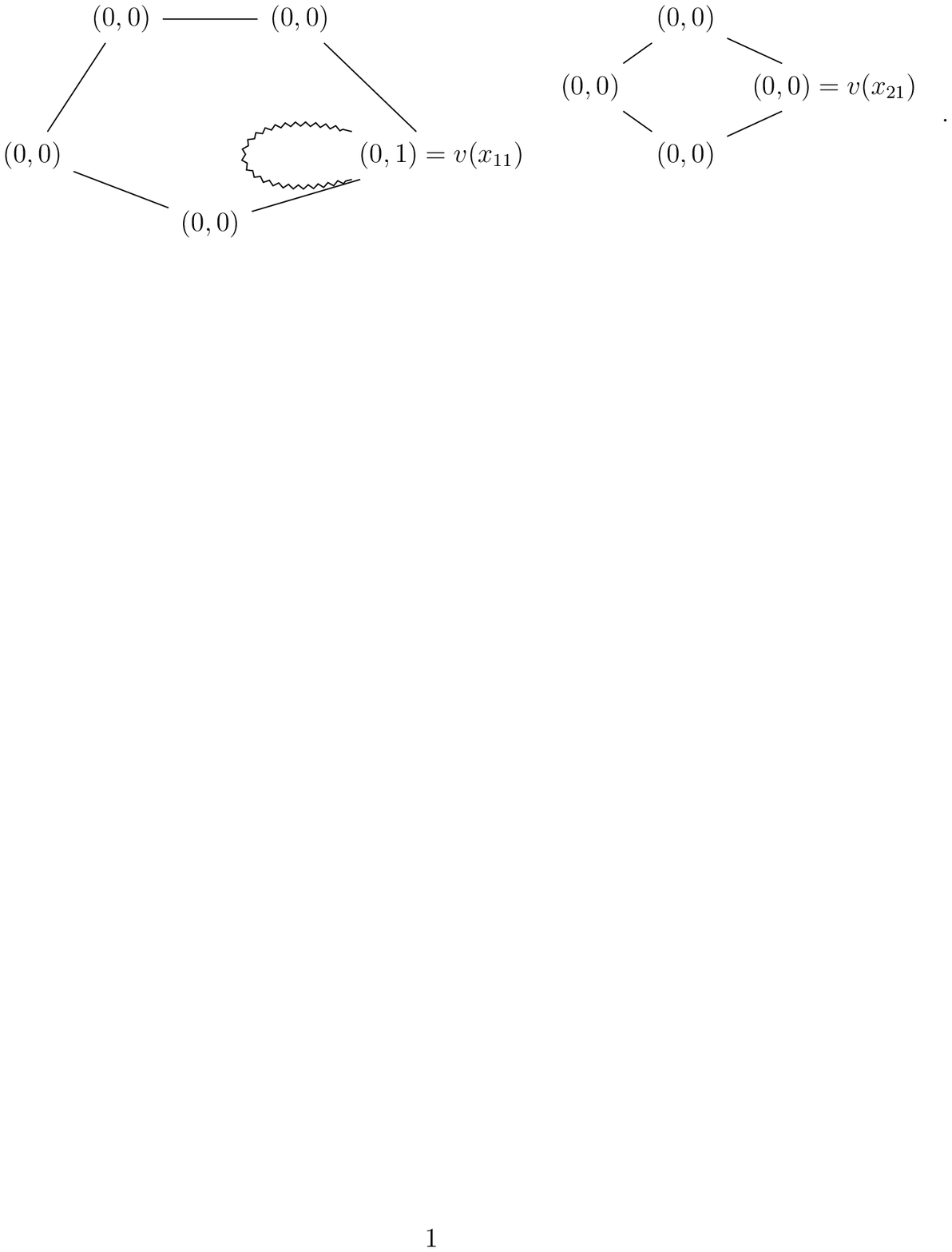}}
\end{equation}
\end{example}

\subsection{Shannon capacity}\label{SSS:ShannonCapacity}

The intuition behind the following invariant comes from viewing an Inca foam as an information carrier. More formally,
think of $M=(G,S,\phi)$ as a noisy communication channel through which colours in the fundamental quandle $(Q(M),B(M))$ as well as interactions are transmitted from (A)lice to (B)ob \cite{Shannon:56}. When $M$ is noisy and non-perfect, the messages on Bob's end appear corrupted and missing. A natural question can then be raised: What is the amount of non-confusable information that can be received by Bob?

Alice has a Gau{\ss} diagram $M$ with fundamental quandle $(Q(M),B(M))$. 
Alice sends Bob a Gau{\ss} diagram equivalent to $M$ and $k$ colours for $k$ vertices in $G$ (not necessarily distinct). For an interaction:
\begin{equation}
\centering
\raisebox{-0.5\height}{\includegraphics[width=0.15\textwidth]{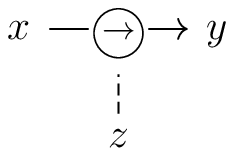}}
\end{equation}
\noindent we say that any \textbf{pair} of elements of the set $\{x,y,z\}$ can be \emph{confused}. Messages which cannot be confused are called \emph{distinct}. In general, if one message can uniquely be recovered from another by using the quandle axioms and an automorphism of $(Q(M),B(M))$, the two messages are said to be \emph{confused}. This same notion was called a \emph{tangle machine computation} in \cite{CarmiMoskovich:15b}.

Let $\textrm{Cap}_k(M)$ denote the number of distinct messages of length $k$ which $M$ admits.

\begin{definition}[Shannon capacity]
The \emph{Shannon capacity} of Gau{\ss} diagram $M$ is:
\begin{equation}
\mathrm{Cap}(M)\ass \sup_{k\in\mathds{N}}\sqrt[k]{\mathrm{Cap}_k(M)}
\end{equation}
\end{definition}

\begin{example}
Consider the Gau{\ss} diagram:
\begin{equation}
\centering
\raisebox{-0.5\height}{\includegraphics[width=0.35\textwidth]{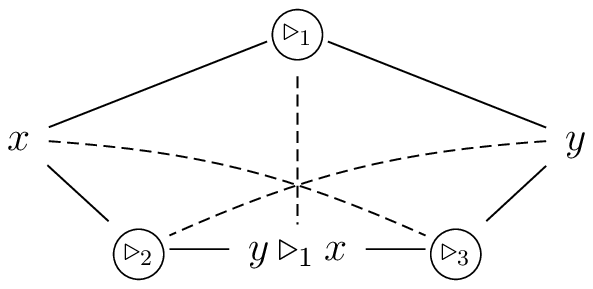}}
\end{equation}
Any two elements of $Q(M)$ are related by an automorphism of $(Q(M),B(M))$, therefore $\mathrm{Cap}_1(M)=1$. A maximal set of distinct messages of length $2$ is $\{xx, xy\}$ and so $\mathrm{Cap}_2(M)=2$. 
 It seems therefore as though $\mathrm{Cap}(M)=\sqrt{2}$.
\end{example}

The definition of the Shannon capacity of a Gau{\ss} diagram mimics that of the Shannon capacity of a graph \cite{Shannon:56}. It is a stable invariant because essentially it is an invariant of the fundamental quandle.

\begin{remark}
A generalization of the above definition would be for Alice to send Bob only partial information about $\phi$, and perhaps even no crossing information at all. 
\end{remark}

\section{Prime decomposition}\label{S:PrimeDecomposition}

To simplify notation and exposition, all Inca foams in this section are taken to be connected. But all constructions and proofs should generalize along the lines of \cite{Hashizume:58} for the multiple component case.

\subsection{The connect sum operation}
The definitions of this section are stated in terms of Gau{\ss} diagrams for convenience, but they apply equally to any of the other diagrammatic formalisms, and indeed to Inca foams, by Theorem~\ref{T:Equivalence}.

A Gau{ss} diagram $M\ass (G,S,\phi)$ is a \emph{connect sum} of $M_1\ass (G,S_1,\phi_1)$ and $M_2\ass (G,S_2,\phi_2)$ if $S_1\cap S_2=\emptyset$ and $\rho|_{S_{1,2}}=\rho_{1,2}$. In this case we write $M\ass M_1\Hash M_2$.

\begin{equation}
\centering
\raisebox{-0.5\height}{\includegraphics[width=0.95\textwidth]{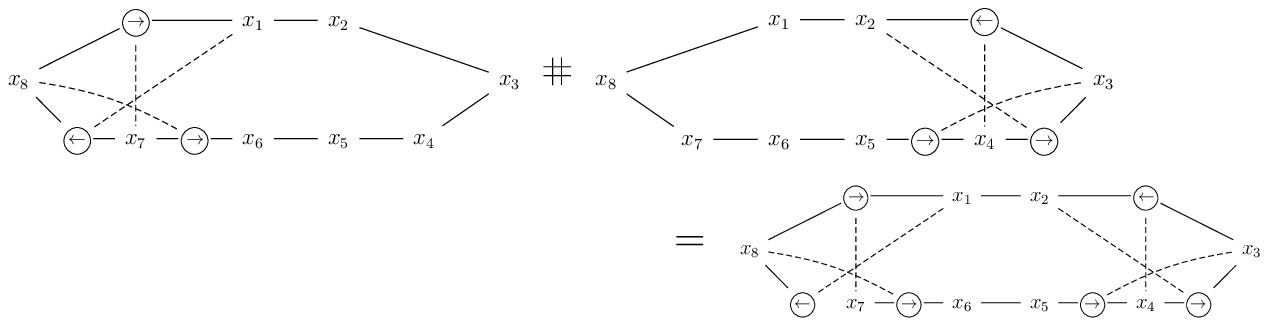}}
\end{equation}

The set of Gau{ss} diagrams with fixed underlying graph $G$ forms a commutative monoid under the connect sum operation. The identity element is the \emph{trivial Gau{ss} diagram} $(G,\emptyset,\phi_\emptyset)$, where $\phi_\emptyset$ denotes the empty function.

In the language of Inca foams, $T=T_1\Hash T_2$ signifies the existence of a \emph{Conway sphere} in $\mathds{R}^4$ which intersects $T$ only at disks at which its spheres meet, which admits a cross-section in a Rosemeister diagram of $T$ which intersects $T$ at two points with all interactions of $T_1$ in the inside and all interactions of $T_2$ on the outside, or vice versa. These spheres can be chosen to project to circles in a corresponding Reidemeister diagram. See Figure~\ref{F:tangle_4dnew}.

\begin{figure}[htb]
\begin{minipage}[t]{1\linewidth}
\centering
    \includegraphics[width=0.7\textwidth]{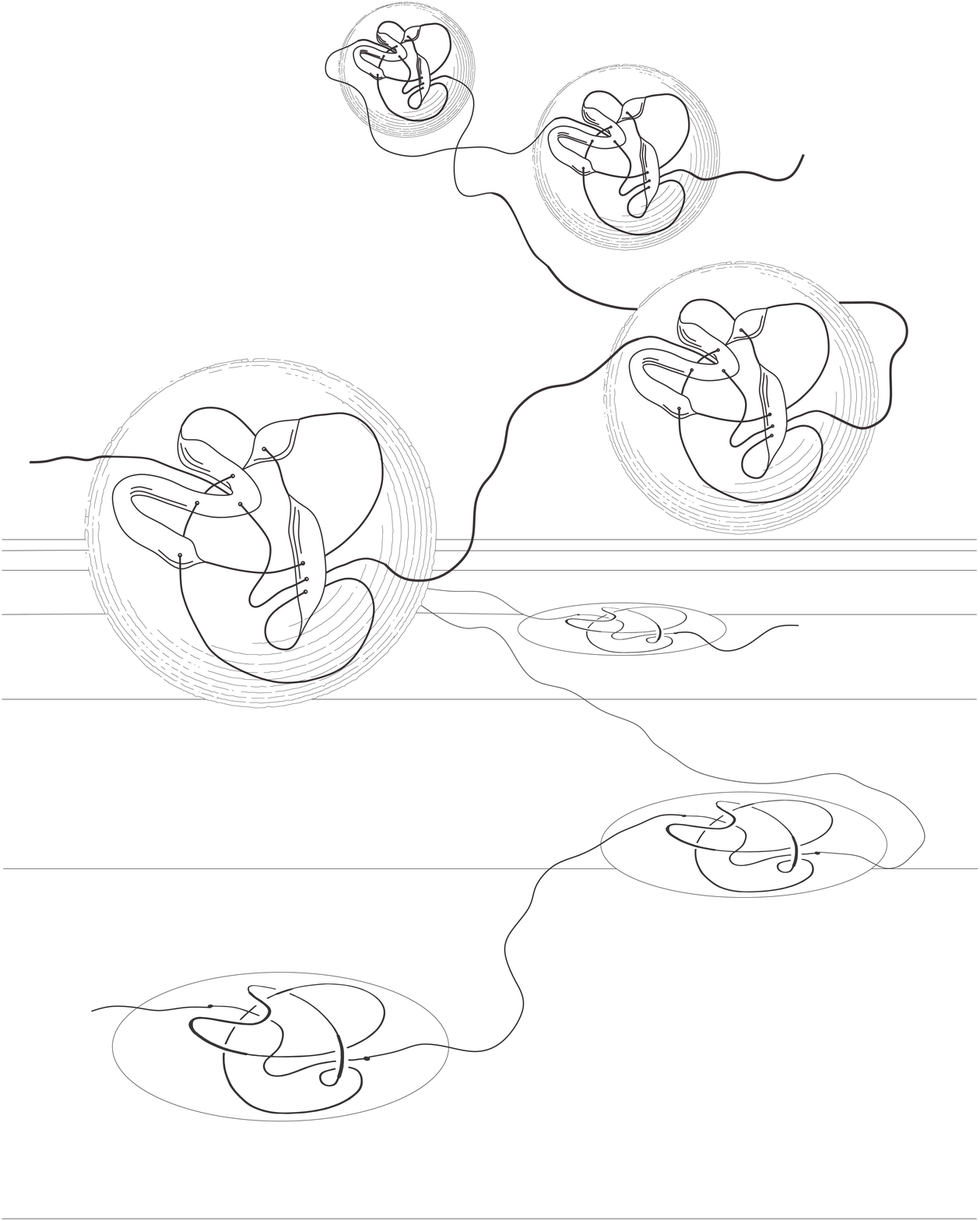}
\end{minipage}
\caption{\label{F:tangle_4dnew}The Roseman diagram for the sphere-and-interval tangle above is represented by a Reidemeister diagram appearing as the `shadow' on the plane underneath. A collection of decomposing spheres projects to a collection of decomposing circles on the Reidemeister diagram below.}
\end{figure}

A nontrivial Gau{\ss} diagram $M$ is \emph{prime} if for any decomposition $M=M_1\Hash M_2$ either $M_1$ or $M_2$ is trivial. By convention, trivial Gau{\ss} diagrams are not prime.

\subsection{Invariant: Prime decomposition}

\begin{theorem}
Prime decomposition is an invariant of a connected nontrivial Inca foam up to permutation of prime factors.
\end{theorem}

\begin{lemma}\label{P:ConnSumWD}
Equivalence classes of connected Inca foams form a commutative monoid under connect-sum.
\end{lemma}

\begin{proof}
Translating into Rosemeister diagrams, let $T= T_1\Hash T_2$ and $S$ be a Conway sphere in $H\simeq \mathds{R}^3$ which induces a given decomposition of an Inca foam in $S^4$. We first show that the connect-sum of trivial diagrams is trivial. If $T_1$ is trivial, then $T_2$ can be shrunk into a tiny ball and $T_1$ can be trivialized inside $T$, so that we see that $T$ is trivial if and only if $T_2$ is trivial. We now show that the connect-sum of a non-trivial diagram with anything else is non-trivial. If $T_1$ is non-trivial, then it has an interaction which cannot be trivialized by Reidemeister moves. This interaction corresponds to a sphere in the sphere-and-interval tangle which does not bound a ball in $\mathds{R}^4$ which is disjoint from the rest of $T$. Connect-summing with $T_2$ happens locally inside a small ball, so it cannot create such a `trivializing ball' in $\mathds{R}^4$--- the Conway sphere $S$ can be chosen to be disjoint from such a trivializing ball, so if it exists in $T$ then it also existed in $T_1$.

Commutativity is proven in the same way- shrink $T_2$ into a small ball, and move it all the way through $T_1$ by ambient isotopy.
\end{proof}

A \emph{prime factorization} of a nontrivial Gau{\ss} diagram $M$ is an expression $M=M_1\Hash M_2\Hash \cdots \Hash M_k$ where $M_1,\ldots M_k$ are prime.


\begin{theorem}[Unique prime factorization]\label{T:UniquePrimeFactorization}
Every Gau{\ss} diagram $M$ has a prime factorization $\mathcal{N}\ass N_1\Hash N_2\Hash\cdots \Hash N_k=M$, which is unique in the following sense: For any prime factorization $\mathcal{N}^\prime\ass N^\prime_1\Hash N^\prime_2\Hash\cdots\Hash N^\prime_k=M$ of $M$ that is topologically equivalent to $\mathcal{N}$, then there exists a permutation $\sigma$ on $k$ elements, and a set $\set{T_1,T_2,\ldots,T_k}$ of unit factors, such that $N_i= N_{\sigma(i)}^\prime\Hash\, T_i$ for all $i=1,2,\ldots,k$.
\end{theorem}

Theorem~\ref{T:UniquePrimeFactorization} follows from the Diamond Lemma, whose hypotheses are satisfied by the following lemma together with finiteness and invariance of the reduced underlying graph $G^{\text{reduced}}$ (Section~\ref{SS:UnderGraph}) which guarantees that the number of prime factors of a Gau{\ss} diagram is finite.

To state the next lemma, define \emph{sub-factorization} $\mathcal{N}^\prime$ of a factorization $\mathcal{N}\ass N_1\Hash N_2\Hash\cdots\Hash N_k$ of Gau{\ss} diagram $M$ to be a factorization of $M$  obtained from $\mathcal{N}$ by factorizing one of its factors $N_i^\prime\Hash N_i^{\prime\prime} = N_i\in \mathcal{N}$.

\begin{lemma}
Any two sub-factorizations $\mathcal{N}^\prime$ and $\mathcal{N}^{\prime\prime}$ of the same factorization $\mathcal{N}$ share a common sub-factorization $\mathcal{N}^{\prime\prime\prime}$.
\end{lemma}

\begin{proof}
We use the topology of sphere-and-interval tangles. Without the limitation of generality, Inca foams are assumed to be \emph{non-split} (there do not exist two disjoint $4$--spheres each containing a nontrivial subtangle of our tangle). The proof is analogous to the proof of unique prime decomposition for knots in $3$--space (\textit{e.g} \cite{BurdeZieschang:03}).

We first establish language.

The converse of connect sum is \emph{cancellation}. To \emph{cancel} a factor $N=(H,S_H,\phi_H)$ in $M=(G,S,\phi)$ is to replace $M$ by a Gau{ss} diagram $M-N\ass (G,S\setminus S_H,\phi|_{S\setminus S_H})$. Topologically, we cancel a factor by replacing each of its spheres in $H$ by an interval connecting its incident segments. For concreteness, parameterizing $S^2$ as the unit sphere on the $xyz$ hyperplane in $\mathds{R}^4$, we replace $S^2$ by $(\cos(t),0,\sin(t),0)$ with $t\in[0,\pi]$, smoothing corners as required. See Figure~\ref{F:Trivialization4d}.

\begin{figure}[htb]
\centering
\includegraphics[width=4.5in]{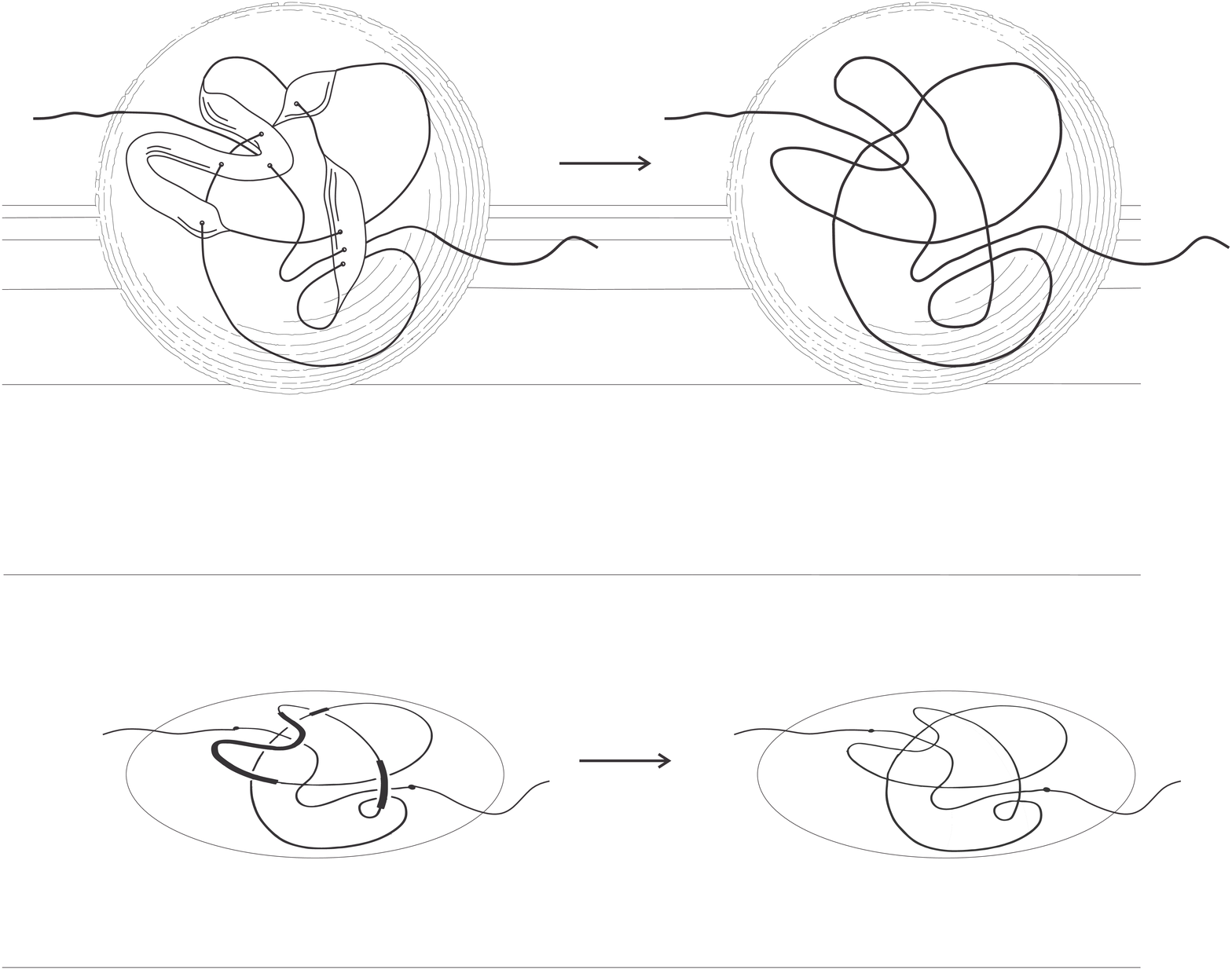}
\caption{\label{F:Trivialization4d}Cancelling a factor.}
\end{figure}

Fix a $3$--dimensional hyperplane $H$ with respect to which we take a Rosemeister diagram $D$ for $K$.

A \emph{system of decomposing spheres} for a sphere-and-interval tangle $K$ is a set of disjoint Conway spheres $S_1,S_2,\ldots,S_k$ embedded in $S^4\simeq \mathds{R}^4\cup \{\infty\}$ bounding $2k$ $4$--balls $B^\mathrm{in}_1,B^\mathrm{in}_2,\ldots,B^\mathrm{in}_k$ and $B^\mathrm{out}_1,B^\mathrm{out}_2,\ldots,B^\mathrm{out}_k$ in $S^4$ correspondingly. If $B_i$ properly contains $4$--balls $B_{\ell(1)},B_{\ell(2)},\ldots,B_{\ell(s)}$ then the \emph{domain} of $S_i$ is defined to be $B^\mathrm{in}_i$ minus the interiors of $B^\mathrm{in}_{\ell(1)},B^\mathrm{in}_{\ell(2)},\ldots,B^\mathrm{in}_{\ell(s)}$. We require that $K\cap \bigcup_{i=1}^k B^{\mathrm{out}}_i$ is trivial, so that all of the `action' takes place inside the domains of $S_1,S_2,\ldots,S_k$.

To subdivide, bisect a decomposing sphere using a $3$--dimensional hyperplane $P\simeq \mathds{R}^3$, separating it into two spheres. For simplicity, we are ignoring the technical details of how to push off the resulting spheres relative to one another, smoothing corners, general position, \textit{etc.}

Let $S_1,S_2,\ldots,S_{m-1}$ be a set of decomposing spheres inducing the factorization $\mathcal{N}$. Without the limitation of generality we may assume that $\mathcal{N}^\prime$ and $\mathcal{N}^{\prime\prime}$ both arise from $\mathcal{N}$ by a single subdivision. If the sub-factorizations $\mathcal{N}^\prime$ and $\mathcal{N}^{\prime\prime}$ arise from bisections of distinct balls $B_i^{\mathrm{in}}$ and $B_j^{\mathrm{in}}$, we can perform both bisections simultaneously to obtain a common refinement $\mathcal{N}^{\prime\prime\prime}$ for both $\mathcal{N}^\prime$ and $\mathcal{N}^{\prime\prime}$. If both sub-factorizations are bisections of the same ball $B_{m-1}^{\mathrm{in}}$, let us take $S_1,S_2,\ldots S_{m-2}, S^\prime_{m-1},S^\prime_{m}$ as the system of decomposing spheres $\mathcal{N}^\prime$, and $S_1,S_2,\ldots S_{m-2}, S^{\prime\prime}_{m-1},S^{\prime\prime}_{m}$ as the system of decomposing spheres $\mathcal{N}^{\prime\prime}$, where $(S_{m-1}^\prime,S_m^\prime)$ is induced by bisecting $S_{m-1}$ along a $3$--dimensional hyperplane $L^\prime$, and $(S_{m-1}^{\prime\prime},S_m^{\prime\prime})$ is induced by bisecting $S_{m-1}$ along a $3$--dimensional hyperplane $L^{\prime\prime}$. Assume general position, and cut along both $L^\prime$ and $L^{\prime\prime}$, pushing off and smoothing as required. Each of $L^\prime$ and $L^{\prime\prime}$ meet $K$ at zero, one, or two points, and in all cases we obtain a new set of decomposing spheres plus some spheres containing trivial factors.  We are working modulo trivial factors, so these trivial factors created along the way may be ignored. We have thus constructed the requisite common refinement.
\end{proof}

\begin{remark}
Neither w-knots nor virtual knots (w-knots without the UC move) have a good notion of prime decomposition \cite{Kim:00,KishinoSatoh:04}. The classical counterexample for virtual knots is Kishino's knot, which is a nontrivial virtual knot both of whose components are trivial (Figure~\ref{F:Kishino}). But, as we have shown, Inca foams suffer no such deficiency. We illustrate with an analogue of Kishino's knot in Figure~\ref{F:Kishino_m1}.
\end{remark}

\begin{figure}
\centering
\includegraphics[width=2in]{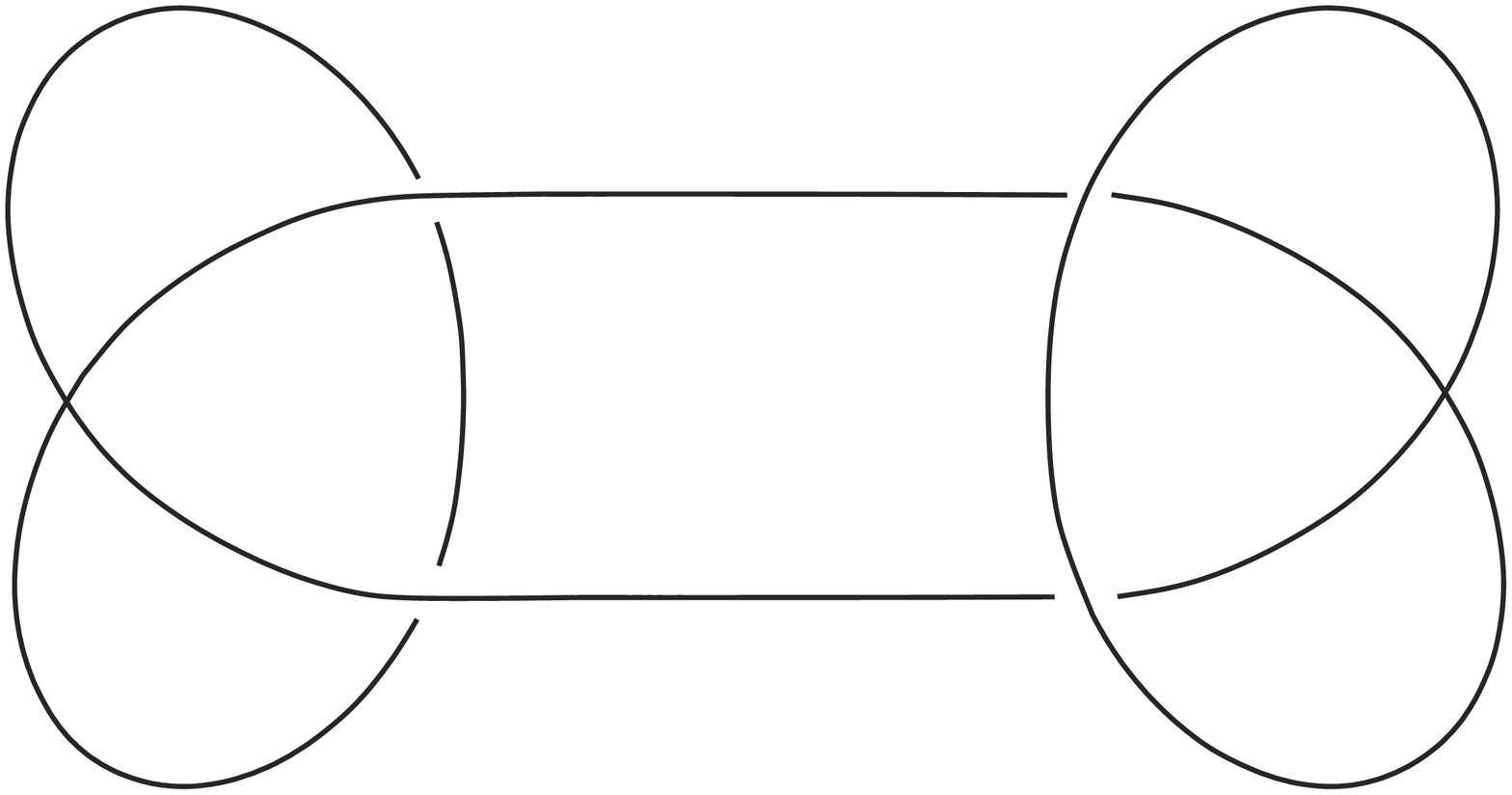}
\caption{\label{F:Kishino}Kishino's knot. A nontrivial connect sum of two trivial virtual knots.}
\end{figure}

\begin{figure}
\centering
\includegraphics[width=5in]{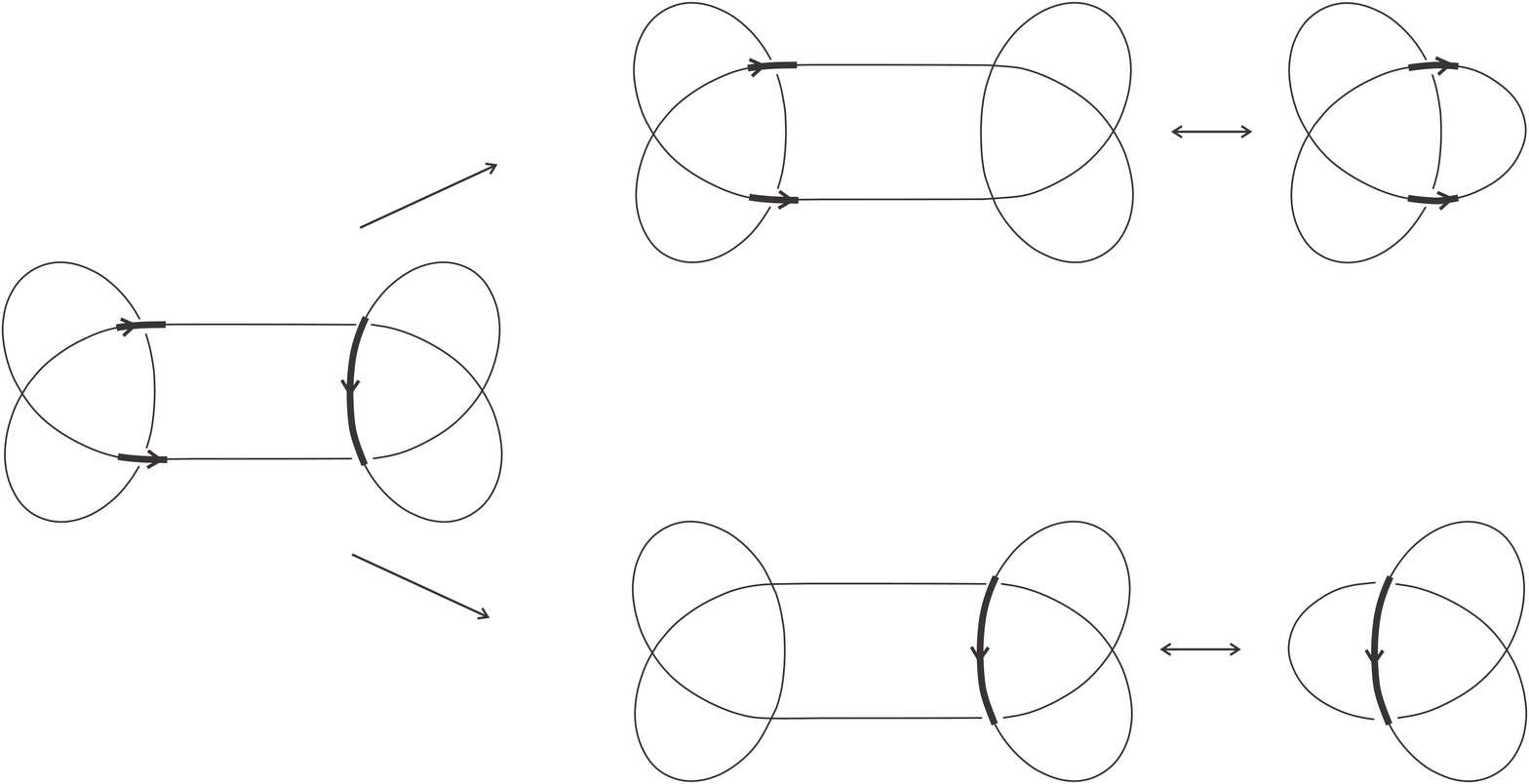}
\caption{\label{F:Kishino_m1}An Inca foam analogue of Kishino's knot is perfectly well-behaved with respect to connect sum.}
\end{figure}

\section*{Acknowledgments}
DM wishes to thank Dror Bar-Natan for useful discussions and for pointing out the topological picture presented in this paper.

\end{document}